\documentclass[leqno]{article}
\pdfoutput=1
\usepackage{amssymb}
\usepackage{amsfonts}
\usepackage{amsmath}
\usepackage{amsfonts}
\usepackage{amsmath}
\usepackage{fancyhdr}
\usepackage{subfig}
\usepackage{graphicx}
\usepackage{caption}
\usepackage{geometry}
\usepackage{latexsym}
\usepackage{url}
\usepackage{hyperref}
\usepackage{mathrsfs,amsmath}
\usepackage{titlesec}
\usepackage{titletoc}
\usepackage{color}
\usepackage{enumitem}
\usepackage{amsthm}
\usepackage{xcolor}

\theoremstyle{plain}
\newtheorem{theorem}{Theorem}[section]

\newtheorem{proposition}[theorem]{Proposition}

\theoremstyle{definition}

\newtheorem{remark}[theorem]{Remark}
\newtheorem{example}{Example}[section]

\numberwithin{equation}{section}

\allowdisplaybreaks

\def\R{{\mathbb R}}
\def\N{{\mathbb N}}

\numberwithin{equation}{section}

\begin{document}

\title{An Efficient Time-splitting Method for the Ehrenfest Dynamics\thanks{Research supported by  NSF grants no. DMS-1348092, DMS-1522184, and DMS-1107291:
RNMS KI-Net, as well as by the Office of the Vice Chancellor for Research and Graduate Education at the University of
Wisconsin-Madison with funding from the Wisconsin Alumni Research Foundation.}}
\author{Di Fang\footnote{%
Department of Mathematics, University of Wisconsin-Madison, Madison, WI
53706, USA (dfang9@wisc.edu).} , Shi Jin\footnote{%
Department of Mathematics, University of Wisconsin-Madison, Madison, WI
53706, USA (sjin@wisc.edu).}, and Christof Sparber\footnote{%
Department of Mathematics, Statistics, and Computer Science, University of
Illinois at Chicago, Chicago, IL 60607, USA (sparber@uic.edu)}}
\date{\today}

\maketitle

\begin{abstract}
The Ehrenfest dynamics, representing a quantum-classical mean-field type coupling, is a widely used approximation in quantum molecular dynamics.
In this paper, we propose a time-splitting method for an Ehrenfest dynamics,
in the form of a nonlinearly coupled Schr\"odinger-Liouville system. We prove
that our splitting scheme is stable uniformly with respect to the semiclassical
parameter, and, moreover, that it preserves a discrete semiclassical limit.
Thus one can accurately compute physical observables using time steps induced only by the classical Liouville equation, i.e., independent of
the small semiclassical parameter - in addition to classical mesh sizes for the Liouville equation. Numerical examples illustrate the validity of our meshing strategy.
\end{abstract}

\section{Introduction}

{\it Ab initio} methods have played a
fundamental role in the numerical simulation of large quantum systems, in particular in quantum molecular dynamics.
Different from classical approaches based on pre-defined potentials, the underlying idea of \textit{ab initio} molecular
dynamics is to compute the forces acting on the nuclei as a feedback of the
electronic structures. This procedure is also known as the
\textquotedblleft on-the-fly" calculation in the chemistry literature (for
detailed reviews, see, e.g., \cite{billing2003quantum, marx2009ab, marx2000ab,
tully1998mixed}). One of the most widely used of these methods is the
so-called \textit{Ehrenfest dynamics}, a mean-field treatment
named in honor of Paul Ehrenfest who was among the first to address the problem of
how to derive classical dynamics from the underlying quantum mechanical equations \cite%
{ehrenfest1927bemerkung}. His idea is to separate the whole system into two
parts: a fast varying, quantum mechanical part (for, say, electrons) and a slowly varying part (for the much heavier nuclei)
in which one can pass to the (semi-)classical limit. In quantum chemistry, this is usually
possible by taking advantage of the large mass difference between electrons and nuclei.

Typically, the Ehrenfest molecular dynamics refers to a Schr\"{o}dinger equation,
coupled with a classical Newtonian flow, cf. \cite{bornemann1996quantum, drukker1999basics, schutte1999singular,
szepessy2011, szepessy2015, tully1998mixed}. The simplest such model reads
\begin{equation}
\label{SN-1}
\left \{
\begin{split}
ih \partial _{t}\psi &=-\frac{h ^{2}}{2}\Delta_{x}\psi +V\left( x,y( t) \right) \psi , \quad \psi(0,x) = \psi_{\rm in}(x),\\
\dot{y} (t)  &= \eta(t), \quad y(0) = y_0,  \\
\dot \eta(t)& = -\nabla _{y}V_{\rm E}( y (t)), \quad \eta(0)=\eta_0 .
\end{split}
\right.
\end{equation}
Here, we denote by $0<h\ll 1 $ a dimensionless rescaled Planck's constant, and by
$\psi= \psi\left( x,t\right)$ with $x\in \mathbb{R}^d$, $t\in \mathbb{R}_{+}$, the wave function of the fast, quantum mechanical degrees of freedom, which is
assumed to be normalized such that $\| {\psi} (\cdot, t)\|_{L^2}=1$ for all $t\geq 0$.
In addition, the slow degrees of freedom are described, for any time $t\in \mathbb{R}_+$, by their classical position $y(t) \in\mathbb{R}^n $ and momentum $\eta(t)\in\mathbb{R}^n$.
We thereby allow for $n\in \mathbb N$ and $d\in \mathbb N$ to be not necessarily equal, depending on the physical application.
Finally, for a given coupling potential $V=V(x,y)\in \mathbb{R}$, the force describing the back-reaction of the quantum part onto the slow degrees of freedom
is given by the gradient in $y\in \mathbb{R}^n$ of the so-called {\it Ehrenfest potential}
\begin{equation*}
V_{\rm E}(y,t)=\int_{\mathbb{R}^{d}}V\left( x,y\right) \left\vert \psi (x,t)\right\vert ^{2}\, dx.
\end{equation*}
Clearly, one obtains a version of Newton's second law for $y(t)$ by eliminating the
momentum variable $\eta(t)$ and writing
\[
\ddot y(t) = -\nabla _{y}V_{\rm E}( y (t)),
\]
instead of the first order Hamiltonian system above.

Regarding the derivation of Ehrenfest dynamics, the majority of literature available today
invokes WKB asymptotics for the slow degrees of freedom, leading to a Hamilton-Jacobi equation which suffers from
the appearance of caustics, see, e.g., \cite{bornemann1996quantum, schutte1999singular}. To circumvent this problem and derive
a semiclassical limit which is valid globally in time, a, by now classical, tool is the
{\it Wigner transform} \cite{wigner1932quantum}. The latter gives rise to
a {\it Liouville equation} for the associated semi-classical phase-space measure (or, Wigner measure) which ``unfolds the caustics", see \cite{wigner4, lions1993mesures,
markowich1993classical, SMM2003}. In the context of Ehrenfest dynamics, such an analysis was carried
out in \cite{Jin:2017bh}. In there, the authors start from a system of \textit{%
time-dependent, self-consistent field equations}, motivated by \cite%
{tdscf1, tdscf3, tdscf2, tdscf4}, and derive (among other things) the following mixed quantum-classical system:
\begin{equation}
\label{SL-1}
\left\{
\begin{split}
& ih \partial _{t}\psi ^{h }=-\frac{h ^{2}}{2}\Delta _{x}\psi
^{h }+\Upsilon ^{h }\left( x,t\right) \psi ^{h }, \quad \psi^h (0,x) = \psi^h_{\rm in}(x) \\
& \partial _{t}\mu ^{h }+\eta \cdot \nabla _{y}\mu ^{h }+F^{h
}\left( y,t\right) \cdot \nabla _{\eta }\mu ^{h }=0,  \quad \mu^h(0,x,\eta) = \mu_{\rm in}(y, \eta).
\end{split}
\right.
\end{equation}%
Here, $\mu ^{h } (\cdot, \cdot, t) \in \mathcal M^+(\mathbb{R}^n_y\times \mathbb{R}^n_\eta)$
denotes the phase-space probability density for the slowly varying degrees of freedom at time $t$, $F^h = - \nabla_y V_{\rm E}$, i.e., the force obtained from the Ehrenfest potential, and
\begin{eqnarray}
\label{potentials}
\Upsilon ^{h }\left( x,t\right) &=&\iint_{\mathbb{R}^{2n}}V\left(
x,y\right) \mu ^{h }\left( y,\eta ,t\right) \, dy\, d\eta .
\end{eqnarray}%
We call this system  the \textit{%
Schr\"{o}dinger-Liouville-Ehrenfest} (SLE) System and from now on represent the dependence on the small semi-classical parameter $h>0$ by superscripts.
Note that the dependence of $\mu^h$ on $h$ stems purely from the forcing through the Ehrenfest potential appearing in the Liouville equation.
The latter is an Eulerian description of the classical Hamiltonian flow. In particular, one formally obtains \eqref{SN-1}, from \eqref{SL-1},
in the case where $\mu$ corresponds to a single particle distribution concentrated on the classical trajectories $(y(t), \eta(t))$, i.e.,
\[
\mu(t,y,\eta) = \delta(y-y(t), \eta-\eta(t)).
\]
Such kind of Wigner measures can be obtained as the classical limit of a particular type of wave functions, called semi-classical wave packets, or coherent states, see \cite{lions1993mesures}.

Given the dispersive nature of Schr\"odinger's equation, the main numerical difficulty for $h \ll 1$
is that one needs to resolve oscillations of frequency of order $\mathcal O(1/h)$ in both time and space, as they are
present in the solution $\psi^h$, see \cite{JMS} for a broad review of this problem. Naively, this requires one to use
time-steps of order $\Delta t= o(h)$ as well as a spatial grid with $\Delta x = o(h)$. However, it was proved in
\cite{Bao:2002fy}, using a Wigner measure
analysis, that for a single linear Schr\"odinger equation,
a time-splitting spectral method can still correctly capture {\it physical observables}, i.e., real-valued quadratic quantities in $\psi^h$, even for time-steps
much larger than $h$. Thus one only needs to
resolve the high frequency oscillations spatially, which is a huge numerical advantage.
For nonlinear Schr\"odinger equations, in general, this is no longer true, as was numerically
demonstrated in \cite{bao2003numerical}.
The SLE system (\ref{SL-1}) is a nonlinearly coupled system, and one therefore expects the same type of problem at first glance. Nevertheless, we shall in the following
develop an efficient numerical method for
the SLE system which allows large
(compared with $h$) computational mesh-sizes in both $y$ and $\eta$ and a large time step for both the Schr\"odinger and the Liouville
equations, while still {\it correctly capturing physical observables}. While large
meshes in $y$ and $\eta$ do not seem so surprising, the possibility of large time steps for solving
the Schr\"odinger equation is far from obvious, due to the nonlinear nature
of the SLE system.

Our numerical algorithm is inspired by, but \textit{different} from the time-splitting method used in \cite{Jin:2017bh}.
Based on a spectral method for the Schr\"odinger equation and an upwind scheme for the Liouville part of \eqref{SL-1},
we shall first prove stability for our algorithm, uniformly in $h$. Furthermore, by utilizing the Wigner
analysis developed in \cite{Jin:2017bh} and adopting it to our particular setting, we shall also
prove that physical observables (which can
be characterized by the moments of the Wigner distribution), are captured
correctly even if $\Delta y$, $\Delta \eta$ and $\Delta t$, i.e., the time step for the entire SLE system, are $\mathcal O(1)$ and thus
\textit{independent of} $h$. To this end, we follow the strategy of \cite{Bao:2002fy}, and prove
that the semi-discretized SLE-system, with $\Delta y$, $\Delta \eta$, $\Delta t$ fixed, converges to the correct semiclassical
limiting system, as $h \to 0$. In this analysis we shall, for simplicity, consider $x$ to be
continuous, since, as already stated above, $\Delta x \to 0$, as $h \to 0$, even for a single linear Schr\"odinger equation.
In summary, our scheme can be seen to be \textit{asymptotic-preserving} in $t$, $y$, and
$\eta$, which is a well-established numerical concept for multi-scale
kinetic equations, cf. \cite{Jin99, Jin-review}. To our knowledge, this is the
\textit{first} work that proves the existence of a {\it global in-time} $h$-\textit{independent} meshing strategy for physical observables associated to 
a \textit{nonlinear} Schr\"odinger-type system.

In this context, we note that the authors of \cite{Carles1, Carles2} study a time-splitting scheme for nonlinear Schr\"odinger equations with 
{\it cubic nonlinearity}. Using a WKB type representation of the solution, they are able to prove a similar asymptotic preserving property. However, the main drawback of 
their method is, that it is only valid before the formation of caustics in the Hamilton-Jacobi equation for the WKB phase function. The system studied in the present paper has a weaker nonlinear structure 
which allows the use of Wigner transformation techniques which are valid for all time $t$.

The rest of this paper is now organized as follows: In Section \ref{sec:scheme} we
present the time-splitting method for the
SLE system and briefly discuss some of the inherent numerical difficulties. The stability, uniformly in $h$,
is then proved in Section \ref{sec:stab} for the fully discretized system. In
Section \ref{sec:wigner}, we shall give a brief review of Wigner transformation methods and
the classical limit of the SLE system. The spatial meshing strategy announced above is then studied in
Section \ref{sec:mesh} by deriving the classical limit of a semi-discrete SLE system. In Section \ref{sec:time}
we focus on the time-discretization and prove that our scheme allows for time-steps independent of $h$.
Finally, Section \ref{sec:examples} presents some numerical examples illustrating our
analytical results.


\section{The time-splitting scheme and its basic properties}\label{sec:scheme}

We shall, from now on, consider the Schr\"{o}dinger-Liouville-Ehrenfest (SLE) System
(\ref{SL-1}) with the following assumption on the coupling potential $V$:
\begin{equation*} \label{a1}
V\in C_{0}^{2}\left( \mathbb{R}_{x}^{d}\times
\mathbb{R}_{y}^{n}\right) \ \text{and }V\left( x,y\right) \geq 0,\forall
\left( x,y\right) \in \mathbb{R}_{x}^{d}\times \mathbb{R}_{y}^{n}, \tag{A1}
\end{equation*}%
where $C^2_{0}$ denotes the set of twice continuously differentiable functions which vanish at infinity together with
all their derivatives.

\begin{remark}
This is the same
assumption as in \cite{Jin:2017bh}, where it is used to furnish a rigorous Wigner analysis of the
self-consistent field equation. Note that,
in particular, it implies $V \in W^{2,\infty }\left( \mathbb{R}_{x}^{d}\times
\mathbb{R}_{y}^{n}\right)$. It is conceivable that the regularity requirement and the decay at infinity can be lowered
at the expense of more technicalities.
The assumption $V(x,y)\geq 0$ is in fact not very restrictive for the potentials
bounded from below. It corresponds to a proper choice of the zero point of the
potential axis.
\end{remark}

We aim for an algorithm which fully utilizes the quantum-classical coupling. Thus,
while it makes sense to use a finer (i.e., smaller than $h$)
spatial discretization in $x$ to solve the Schr\"odinger equation, we want to use much larger (than $h$) meshes in $y$ and $\eta$
when solving the Liouville equation. That this is indeed possible is not obvious, since the potential $\Upsilon ^{h }\left( x,t\right) $
appearing in the Schr\"{o}dinger equation is time-dependent, and moreover nonlinearly coupled to the Liouville equation (hence it inherits the
computational error obtained from discretizing in $y$ and $\eta$).

\begin{remark}\label{remcom}
In our discussion, we will only consider compactly supported initial
data $\psi^h_{\rm in}$, $\mu_{\rm in}$, in order to simulate the SLE system problem based on an
infinitely large spatial domain within a sufficiently large, but finite box with periodic
boundary conditions.
\end{remark}


\subsection{A new time-splitting scheme for the SLE system}\label{sec:tsscheme}

In order to describe our scheme, we henceforth assume that we are given a sufficiently small $\Delta x\sim \mathcal O(h)$, used to solve the quantum mechanical part of \eqref{SL-1}, while
the larger grid meshes $\Delta y,\Delta \eta \sim \mathcal O\left( 1\right)  $ are applied for the classical
part. With this in mind, let
\begin{equation*}
J=\frac{d-c}{\Delta y},\ K=%
\frac{\beta -\alpha }{\Delta \eta }, \ M=\frac{b-a}{\Delta x}, \ y_{j}=c+j\Delta y,\ \eta _{k}=\alpha
+k\Delta \eta, \ x_{j}=a+j\Delta x.
\end{equation*}

The time-splitting spectral scheme can then described as
follows: From time $t=t_{n}=n\Delta t$ to $t=t_{n+1}=\left( n+1\right)
\Delta t$, with $\Delta t$ given, the SLE system is solved in two steps. First, solve
\begin{equation}\label{step1}
\left\{
\begin{split}
& ih \partial _{t}\psi ^{h }  =-\frac{h ^{2}}{2}\Delta _{x}\psi
^{h },   \\
& \partial _{t}\mu ^{h }= -\eta \cdot \nabla _{y}\mu ^{h }-F^{h
}\left( y,t\right) \cdot \nabla _{\eta }\mu ^{h },
\end{split}
\right.
\end{equation}%
from $t=t_{n}$ to an intermediate time $t_{\ast }$. Then, solve
\begin{equation}\label{step2}
\left\{
\begin{split}
& ih \partial _{t}\psi ^{h }=\Upsilon^{h} \left( x,t\right) \psi
^{h },  \\
& \partial _{t}\mu ^{h }=0,
\end{split}
\right.
\end{equation}
with initial data obtained from Step 1, to obtain the solution at time $t=t_{n+1}$.

In (\ref{step1}), the Schr\"odinger equation will be discretized
in space by a spectral method and integrated in time exactly using a Fast Fourier Transform.
The Liouville equation can be solved either by a spectral method, or by a finite difference
(e.g., upwind) scheme in space, and then marching the corresponding ODE system forward in time.
An advantage of our splitting method is that in the second step, $\Upsilon ^{h }\left(
x,t\right) $ defined in (\ref{potentials}) is indeed {\it independent} of time, since obviously $\mu ^{h }$ is.
In view of this, the time integration in \eqref{step2} can also be solved exactly, which yields
\begin{equation*}
\psi _{j}^{h ,n+1}=\exp \left( -\frac{i}{h }\Upsilon ^{h
}\left( x_{j},t_{\ast }\right) \Delta t\right) \psi _{j}^{h ,\ast }.
\end{equation*}

For the convenience of our later discussions, we shall now state our numerical scheme using an upwind
spatial discretization of $\mu $ in more detail: The problem is solved in one spatial dimension $d=n=1$ from time
$t=t_{n}$ to time $t=t_{n+1}$ using the following two steps:

In the first step, we solve
\begin{equation}\label{ts1}
\left\{
\begin{split}
& ih \partial _{t}\psi ^{h }=-\frac{h ^{2}}{2}\partial
_{xx}\psi ^{h }, \\
&\frac{d}{dt}\mu _{jk}^{h } =-\eta _{k}\left( D_{y}\mu
^{h }\right) _{jk}-F_{j}^{h }\left( D_{\eta }\mu ^{h }\right)
_{jk},%
\end{split}
\right.
\end{equation}%
where both $D_{y}\mu ^{h }$ and $D_{\eta }\mu ^{h }$ represent the
numerical derivatives in our algorithm, which are treated using a standard \textit{conservative} (for example, the upwind type) discretization. 
To solve the Liouville equation we shall we apply a forward-in-time Euler scheme for the time discretization.
Explicitly, we thus have
\begin{equation} \label{upwind1}
\left\{
\begin{split}
& \psi _{j}^{h ,\ast }=\frac{1}{{M}}\sum\limits_{\ell=-{{M}}/2}^{{M}/2-1}e^{-ih
\omega _{\ell}^{2}/2}\hat{\psi}_{\ell}^{h ,n}e^{i\omega _{\ell}\left(
x_{j}-a\right) },\quad  j=0,\dots ,{M}-1 , \\
& \frac{\mu _{jk}^{h ,\ast }-\mu _{jk}^{h ,n}}{\Delta t}=-\eta
_{k}\left( D_{y}\mu ^{h ,n}\right) _{jk}-F_{j}^{h ,n}\left(
D_{\eta }\mu ^{h ,n}\right) _{jk},
\end{split}%
\right.
\end{equation}%
where ${w_\ell = \frac{2 \pi \ell}{b-1} }$ and, for the upwind spatial discretization,
\begin{equation*}
\left\{
\begin{split}
& \eta_{k}\left( D_{y}\mu ^{h ,n}\right) _{jk}
= \frac{1}{2}\left({\eta_{k}+\left\vert \eta _{k}\right\vert }\right) \frac{\mu _{jk}^{h,n }-\mu_{j-1,k}^{h,n }}{\Delta y}
+\frac{1}{2}(\eta_{k}-\left\vert \eta _{k}\right\vert )\frac{\mu _{j+1,k}^{h,n }-\mu_{j,k}^{h,n }}{\Delta y} ,
\\
& F_{j}^{h ,n}\left( D_{\eta }\mu ^{h,n }\right) _{jk}
=\frac{1}{2}(F_{j}^{h,n }+\left\vert F_{j}^{h,n }\right\vert ) \frac{ \mu_{jk}^{h,n }-\mu _{j,k-1}^{h }}{\Delta \eta}
+\frac{1}{2}(F_{j}^{h,n }-\left\vert F_{j}^{h,n }\right\vert ) \frac{ \mu_{j,k+1}^{h,n }-\mu _{jk}^{h,n }}{\Delta \eta}.
\end{split}
\right.
\end{equation*}

The second step is then given by
\begin{equation}\label{ts2}
\left\{
\begin{split}
& ih \partial _{t}\psi ^{h }=\Upsilon _{d}^{h }\left(
x,t\right) \psi ^{h }, \\
& \frac{d}{dt}\mu _{jk}^{h }=0,
\end{split}%
\right.
\end{equation}%
where $\Upsilon_{d}^{h }\left( x,t\right) $ is the quadrature approximation of $%
\Upsilon ^{h }\left( x,t\right) $.
Thus, we explicitly have
\begin{equation} \label{upwind2}
\begin{split}
\psi _{j}^{h ,n+1}=\exp \left(-i\Upsilon _{d}^{h ,\ast }\left(
x_{j}\right) {{ \Delta t}} /h \right) \psi _{j}^{h ,\ast } , \quad  \mu _{jk}^{h ,n+1}=\mu _{jk}^{h ,\ast }%
\end{split}%
\end{equation}%
where
\begin{align*}
\Upsilon _{d}^{h ,\ast }\left( x\right)
 =\sum\limits_{j=0}^{J-1}\sum\limits_{k=0}^{K-1}V\left( x,y_{j}\right) \mu
_{jk}^{h ,\ast }\Delta y\Delta \eta
 =\sum\limits_{j=0}^{J-1}\sum\limits_{k=0}^{K-1}V\left( x,y_{j}\right) \mu
_{jk}^{h ,n+1}\Delta y\Delta \eta ,
\end{align*}
which can be viewed as a trapezoidal rule for $\mu$ with compact support, cf. Remark \ref{remcom}.

\begin{remark}
It is straightforward to obtain an algorithm second order in time using the Strang splitting, which
is omitted here.
\end{remark}


\subsection{Conservation property and stability of the scheme}\label{sec:stab}

We shall now prove the stability of the scheme given by (\ref{upwind1}) and (\ref{upwind2}). To this end, let $\mathbf{\psi =}\left( \psi _{0},\ldots ,\psi _{M-1}\right) ^{T}$. Let $%
\left\Vert \cdot \right\Vert _{L^{2}}$ and $\left\Vert \cdot \right\Vert
_{\ell^2}$ be the usual $L^{2}$ and $\ell^2$ norm on the interval $\left(
a,b\right) $ respectively, i.e.%
\begin{equation}  \label{normdef}
\left\Vert {{\phi}} \right\Vert _{L^{2}}=\left(\int_{a}^{b}\left\vert  {{\phi}}
\left( x\right) \right\vert ^{2}\, dx\right)^{1/2},\quad
\left\Vert \mathbf{\psi }\right\Vert
_{\ell^2}=\left(\frac{b-a}{M}\sum\limits_{j=0}^{M-1}\left\vert \psi
_{j}\right\vert ^{2}\right)^{1/2}.
\end{equation}%
Notice that, for any periodic function $f$, the equality%
\begin{equation} \label{norm}
\left\Vert f_{\rm I}\right\Vert^2 _{L^{2}}=\left\Vert f\right\Vert^2 _{\ell^2}=%
\frac{b-a}{M}\sum\limits_{j=0}^{M-1}\left\vert f\left( x_{j}\right)
\right\vert ^{2}
\end{equation}%
holds, where $f_{\rm I}$ denotes the trigonometric interpolant of $f$ on $%
\left\{ x_{0},x_{1},\ldots ,x_{M}\right\} $, i.e.%
\begin{equation*}
f_{\rm I}\left( x\right) =\frac{1}{M}\sum\limits_{\hat{\jmath}=-\frac{M}{2}}^{%
\frac{M}{2}-1}\hat{f}_{\hat{\jmath}}e^{i\omega _{\hat{\jmath}}\left(
x-a\right) },\quad \omega _{\hat{\jmath}}=\frac{2\pi \hat{\jmath}}{b-a}. 
\end{equation*}%
Using this we can prove the following theorem.

\begin{theorem} \label{mass cons}
The time-splitting spectral scheme conserves the mass. More precisely, it holds%
\begin{equation*}
{\|\mathbf{\psi }^{h ,n}\|} _{\ell^2}={\| \psi^{h, 0}\|} _{\ell^2},\quad  {\| \psi _{\rm I}^{h
,n}\| }_{L^{2}}={\|\psi_{\rm I}^{h ,0}\|}
_{L^{2}},\quad  \text{ for $n=1,2,\dots ,$}
\end{equation*}%
{{where, as before, $\psi _{\rm I}^{h,n}$} denotes the trigonometric interpolant of $\psi^{h
,n}$}. In addition,
\begin{equation*} 
\sum\limits_{j=0}^{J-1}\sum\limits_{k=0}^{K-1}%
\mu _{jk}^{h, n }=\sum\limits_{j=0}^{J-1}\sum\limits_{k=0}^{K-1}%
\mu_{jk} ^{h,0 }.
\end{equation*}
\end{theorem}

\begin{proof} First note that the last identity for $ \mu _{jk}^{h, n }$ is a straightforward
consequence of the fact that the discretized derivatives $D_{y} \mu$ and $D_{\eta} \mu$ are conservative.

It suffices to prove the first identity stated above due to (\ref{norm}). Noting our numerical algorithm (\ref{upwind1}), (\ref{upwind2}) and the definition of the norms (\ref{normdef}), one computes
\begin{eqnarray*}
\frac{1}{b-a}\left\Vert \mathbf{\psi }^{h ,n+1}\right\Vert _{\ell^2}^{2}
&=&\frac{1}{M}\sum\limits_{j=0}^{M-1}\left\vert \psi _{j}^{h
,n+1}\right\vert ^{2}=\frac{1}{M}\sum\limits_{j=0}^{M-1}\left\vert \exp
\left( -\frac{i}{h }\Upsilon _{d}^{h }\left( x_{j},t_{\ast
}\right) \Delta t\right) \psi _{j}^{h ,\ast }\right\vert ^{2} \\
&=&\frac{1}{M}\sum\limits_{j=0}^{M-1}\left\vert \psi _{j}^{h ,\ast
}\right\vert ^{2}=\frac{1}{M}\sum\limits_{j=0}^{M-1}\left\vert \frac{1}{M}%
\sum\limits_{\hat{\jmath}=-\frac{M}{2}}^{\frac{M}{2}-1}e^{-ih  {{ \Delta t}}\omega _{%
\hat{\jmath}}^{2}/2}\hat{\psi}_{\hat{\jmath}}^{h ,n}e^{i\omega _{\hat{%
\jmath}}\left( x_{j}-a\right) }\right\vert ^{2} \\
&=&\frac{1}{M}\sum\limits_{j=0}^{M-1}\left( \frac{1}{M^{2}}\sum\limits_{p=-%
\frac{M}{2}}^{\frac{M}{2}-1}\sum\limits_{q=-\frac{M}{2}}^{\frac{M}{2}%
-1}e^{ih  {{ \Delta t}} \left( \omega _{p}^{2}-\omega _{q}^{2}\right) /2}\overline{%
\hat{\psi}_{p}^{h ,n}}\hat{\psi}_{q}^{h ,n}e^{i\left( \omega
_{q}-\omega _{p}\right) \left( x_{j}-a\right) }\right).
\end{eqnarray*}
Changing the order of summation, this is equal to
\begin{eqnarray*}
\frac{1}{b-a}\left\Vert \mathbf{\psi }^{h ,n+1}\right\Vert _{\ell^2}^{2}
&=&\frac{1}{M^{2}}\sum\limits_{p=-\frac{M}{2}}^{\frac{M}{2}%
-1}\sum\limits_{q=-\frac{M}{2}}^{\frac{M}{2}-1}e^{ih  {{ \Delta t}} \left( \omega
_{p}^{2}-\omega _{q}^{2}\right) /2}\overline{\hat{\psi}_{p}^{h ,n}}\hat{%
\psi}_{q}^{h ,n}\left( \frac{1}{M}\sum\limits_{j=0}^{M-1}e^{i\left(
\omega _{q}-\omega _{p}\right) \left( x_{j}-a\right) }\right) \\
&=&\frac{1}{M^{2}}\sum\limits_{\hat{\jmath}=-\frac{M}{2}}^{\frac{M}{2}%
-1}\left\vert \hat{\psi}_{\hat{\jmath}}^{h ,n}\right\vert ^{2}=\frac{1}{%
M^{2}}\sum\limits_{\hat{\jmath}=-\frac{M}{2}}^{\frac{M}{2}-1}\left\vert
\sum\limits_{j=0}^{M-1}\psi _{j}^{h ,n}e^{-i\omega _{\hat{\jmath}}\left(
x_{j}-a\right) }\right\vert ^{2} \\
&=&\frac{1}{M^{2}}\sum\limits_{\hat{\jmath}=-\frac{M}{2}}^{\frac{M}{2}%
-1}\left( \sum\limits_{p=0}^{M-1}\sum\limits_{q=0}^{M-1}\overline{\psi
_{p}^{h ,n}}\psi _{q}^{h ,n}e^{i\omega _{\hat{\jmath}}\left(
x_{q}-x_{p}\right) }\right),
\end{eqnarray*}
where the second equality of the above follows from the fact that%
\begin{equation*}
\frac{1}{M}\sum\limits_{j=0}^{M-1}e^{i\left( \omega _{q}-\omega _{p}\right)
\left( x_{j}-a\right) }=\frac{1}{M}\sum\limits_{j=0}^{M-1}e^{i2\pi \left(
q-p\right) j/M}=\left\{
\begin{array}{cc}
0, & q-p\neq mM \\
1, & q-p=mM%
\end{array}%
\right. ,\ m\in \mathbb Z.
\end{equation*}%
Similarly, by changing the order of summation again, we arrive at
\begin{eqnarray*}
\frac{1}{b-a}\left\Vert \mathbf{\psi }^{h ,n+1}\right\Vert _{\ell^2}^{2}
&=&\frac{1}{M}\sum\limits_{p=0}^{M-1}\sum\limits_{q=0}^{M-1}\overline{\psi
_{p}^{h ,n}}\psi _{q}^{h ,n}\left( \frac{1}{M}\sum\limits_{\hat{%
\jmath}=-\frac{M}{2}}^{\frac{M}{2}-1}e^{i\omega _{\hat{\jmath}}\left(
x_{q}-x_{p}\right) }\right) \\
&=&\frac{1}{M}\sum\limits_{j=0}^{M-1}\left\vert
\psi _{j}^{h ,n}\right\vert ^{2}=\frac{1}{b-a}\left\Vert \mathbf{\psi }%
^{h ,n}\right\Vert _{\ell^2}^{2},
\end{eqnarray*}%
where the following identity has been used%
\begin{equation*}
\frac{1}{M}\sum\limits_{\hat{\jmath}=-\frac{M}{2}}^{\frac{M}{2}-1}e^{i\omega
_{\hat{\jmath}}\left( x_{q}-x_{p}\right) }=\frac{1}{M}\sum\limits_{\hat{%
\jmath}=-\frac{M}{2}}^{\frac{M}{2}-1}e^{i2\pi \left( q-p\right) \hat{\jmath}%
/M}=\left\{
\begin{array}{cc}
0, & q-p\neq mM \\
1, & q-p=mM%
\end{array}%
\right. ,\ m\in \mathbb Z.
\end{equation*}
\end{proof}

\begin{remark}
Theorem \ref{mass cons} implies that the scheme is stable uniformly in $h$, provided the positivity of $\mu$ under the following CFL condition, cf. \cite{leveque1992numerical}:
\begin{equation} \label{cfl}
\max_k|\eta_k|\frac{\Delta t}{\Delta y}+\left\Vert \partial _{y}V\right\Vert _{L^{\infty }}\frac{\Delta t}{\Delta \eta} \leq 1.
\end{equation}
\end{remark}


\section{Classical limit of the SLE system}\label{sec:wigner}

As a preparatory step to the discussion of Section \ref{sec:mesh}, we will now briefly review the results of \cite{Jin:2017bh} concerning the classical limit
(via Wigner transforms) of the SLE system as $h \to 0$.

\subsection{Wigner transform and Wigner measure}

Let us first recall that $h $-scaled Wigner transform associated to
any continuously parametrized family $f^h \equiv \{ f^h \}_{0\leq h \leq 1}\in L^{2}\left( \mathbb{R}^{d}\right) $ is given by, cf. \cite{wigner4, lions1993mesures,
markowich1993classical, SMM2003}:
\begin{equation*}
w^{h }[ f^h ] \left( x,\xi \right) =\frac{1}{\left( 2\pi
\right) ^{d}}\int_{\mathbb{R}^{d}}f^h\left( x-\frac{h }{2}y\right)
\overline{f^h}\left( x+\frac{h }{2}y\right) e^{i\xi \cdot y}\, dy.
\end{equation*}%
By Plancherel's Theorem and a change of variables one easily finds
\begin{equation*}
\left\Vert w^{h }[ f^h ] \right\Vert _{L^{2}\left( \mathbb{R}%
^{2d}\right) }=\frac{1}{\left( 2\pi \right) ^{\frac{d}{2}}h ^{d}}%
\left\Vert f^h\right\Vert _{L^{2}\left( \mathbb{R}^{d}\right) }^{2}.
\end{equation*}%
The real-valued function $w^h(x,\xi) $ acts as a quantum mechanical analogue for classical phase-space distributions. However, $w^h (x,\xi)\not \geq0$ in general.

It has been proved in \cite{lions1993mesures}, that if the family of functions $f^h=\{ f^h \}_{0\leq h \leq 1}$ is uniformly bounded in $L^2(\R^d)$ as $h \to 0_+$, i.e., if
\[\sup_{0< h \leq 1} \| f^h \|_{L_x^2} \leq C,\]  then
the set of Wigner functions $\{w^h\}_{0<h\leq 1} $ is uniformly bounded in $\mathcal A'$. The latter is the dual of the following Banach space
\[
\mathcal A(\R^d_x \times \R^d_\xi):= \{ \chi \in C_0(\R^d_x \times \R^d_\xi)\, : \, (\mathcal F_{\xi } \chi)(x,z) \in L^1(\R^d_z; C_0(\R^d_x)) \},
\]
where $C_0(\R^d)$ denotes the space of continuous functions vanishing at infinity and $\mathcal F_\xi$ denotes the Fourier transform with respect to the velocity $\xi$, only. 
More precisely, one finds that for any test function $\chi \in \mathcal A(\R^d_x \times \R^d_\xi)$,
\[
| \langle w^h, \chi \rangle | \leq \frac{1}{(2\pi)^d} {\| \chi \|}_{\mathcal A} {\| f^h \|}^2_{L^2} \leq {\rm const.},
\]
uniformly in $h$. Thus, up to extraction of sub-sequences $\{h_n\}_{n\in \N}$, with $h_n \to 0_+$ as $n \to \infty$, there exists a limiting object
$w^0\equiv w \in \mathcal A'(\R^d_x \times \R^d_\xi)$ such that
$$
w^h \stackrel{h\rightarrow 0_+
}{\longrightarrow}  w \quad \text{in $\mathcal A'(\R^d_x \times \R^d_\xi) {\rm w-\ast}$} .
$$
It turns out that the limit is in fact a non-negative, bounded Borel measure on phase-space $w \in \mathcal M^+(\R^d_x \times \R^d_p)$, called the {\it Wigner measure}
of $f^h$.


\subsection{The Classical limit of the SLE system}

Let $\psi ^{h }$ and $\mu ^{h }$ be the solution of the SLE system
(\ref{SL-1}) and denote the Wigner function of $\psi ^{h }\left( x,t\right) $ by
\begin{equation*}
w^{h }\left( x,\xi ,t\right) =w^{h }[ \psi ^{h }\left(
\cdot ,t\right) ] \left( x,\xi \right).
\end{equation*}%
A straightforward computation shows that the {\it position density} associated to $\psi^h\in L^2(\R^d)$ can be computed via
\begin{equation*}
{{\rho}}^{h }\left( x,t\right) :=\left\vert \psi ^{h }\left( x,t\right)
\right\vert ^{2}=\int_{\mathbb{R}^{d}}w^{h }\left( x,\xi ,t\right) d\xi,
\end{equation*}
where we recall, that due to our normalization,
\[
\int_{\R^d} {{\rho}}^{h }\left( x,t\right) \, dx = \iint_{\mathbb{R}^{2d}}w^{h }\left( x,\xi ,t\right) d\xi\, dx = 1.
\]
Moreover, by taking higher order moments in $\xi$ one (formally) finds
the {\it current density}%
\begin{equation*}
j^{h }\left( x,t\right) :=h \mathop{\rm Im}\left( \overline{\psi
^{h }\left( x,t\right) }\nabla \psi ^{h }\left( x,t\right) \right)
=\int_{\mathbb{R}^{d}}\xi w^{h }\left( x,\xi ,t\right) d\xi ,
\end{equation*}%
and the {\it kinetic energy density}%
\begin{equation*}
\kappa^{h }\left( x,t\right) :=\frac{h ^{2}}{2}\left\vert \nabla \psi
^{h }\left( x,t\right) \right\vert ^{2}=\int_{\mathbb{R}^{d}}\frac{1}{2}\left\vert \xi \right\vert ^{2} w^{h }\left( x,\xi
,t\right) d\xi .
\end{equation*}%
\begin{remark}
In order to make these computations rigorous, the integrals on the r.h.s. have to be understood in an appropriate sense, since $w^h \not \in L^1(\R^m_x\times \R^m_\xi)$ in general,
see \cite{lions1993mesures} for more details.
\end{remark}

After Wigner transforming the Schr\"odinger equation, one finds that $w^{h }\left( x,\xi ,t\right)$ satisfies the following nonlocal kinetic equation (see, e.g., \cite{lions1993mesures}):
\begin{equation*}
\partial_{t} w^{h }+\xi \cdot \nabla _{x}w^{h }+\Theta ^{h }[
\Upsilon ^{h }] w^{h}=0,\quad  w^{h }(0,x, \xi)=w_{\rm in}^{h}\left( x,\xi \right) ,
\end{equation*}
where $w_{\rm in}^{h}\equiv w^{h}[\psi^h_{\rm in}]$ and
\begin{equation}\label{theta}
\left(\Theta ^{h }[ \Upsilon ^{h }] w^{h }\right)\left( x,\xi,t\right) =
\frac{i}{h (2\pi)^d } \int_{\mathbb{R}^{d}}
\left( \Upsilon ^{h }\left( x+\frac{h }{2}z ,t\right) -\Upsilon^{h }\left( x-\frac{h }{2}z ,t\right) \right)
\widehat{w}^{h }( x,z ,t) e^{iz \cdot \xi }dz
\end{equation}
with $\widehat{w}^{h}$ denoting the Fourier transformation of $w^h$ w.r.t. the second variable only.

Now, in order to utilize the weak-$\ast$ compactness properties of the Wigner function, we shall impose from now on that the initial mass
and the initial kinetic energy are uniformly bounded with respect to $h $, i.e.,
\begin{equation*} \label{a2}
\sup\limits_{0<h \leq 1} \left({{\rho}}^{h}\left( x,0\right) +\kappa^{h }\left(x,0\right)  \right) \leq \text{const.} \tag{A2}
\end{equation*}%
\begin{remark}
In other words, we assume that
\[
\sup\limits_{0<h \leq 1} \left( \left\vert \psi_{\rm in}
^{h }(x) \right \vert^2+ \frac{h ^{2}}{2}\left\vert \nabla \psi _{\rm in}^{h } (x)\right \vert^2 \right) \leq \text{const.}
\]
This assumption is easily satisfied by initial data of WKB type, or by semi-classical wave packets.
\end{remark}

It is proved in \cite{Jin:2017bh} that these uniform bounds on the initial mass and kinetic energy are propagated by the SLE system \eqref{SL-1}, which in turn
implies that for all times $t\in \R_+$, the wave function $\psi^h(\cdot,t)$ is:
\begin{enumerate}
\item uniformly bounded in $L^{2}\left( \mathbb{R} ^{d}\right) $ as $h
\rightarrow 0^{+}$, i.e.%
\begin{equation*}
\underset{0<h \leq 1}{\sup }\left\Vert \psi^{h }(\cdot, t)\right\Vert
_{L_{x}^{2}}\leq C_{1},
\end{equation*}
\item $h $-oscillatory, i.e.%
\begin{equation*}
\underset{0<h \leq 1}{\sup }\left\Vert h \nabla _{x} \psi^{h
}(\cdot, t)\right\Vert _{L_{x}^{2}}\leq C_{2},
\end{equation*}%
where $C_{1}$ and $C_{2}$ are some constants independent of $0<h <1$.
\end{enumerate}

In particular this implies the existence of a limiting Wigner measure $\nu (\cdot, \cdot, t)\in \mathcal{M}^{+}( \mathbb{R}%
_{x}^{d}\times \mathbb{R}_{\xi }^{d}) $, such that for all $T>0$
\begin{equation*}
w^{h }\left[ \psi^{h }\right] \stackrel{h\rightarrow 0_+
}{\longrightarrow} \nu\ \ \text{in }L^\infty([0,T];\mathcal{A}^{\prime }\left( \mathbb{R}%
_{x}^{d}\times \mathbb{R}_{\xi }^{d})\right)\text{w--}\ast ,
\end{equation*}%
up to the extraction of subsequences. Moreover, on the same time-interval, one has
\begin{equation*}
\left\vert \psi ^{h }(x,t) \right\vert ^{2}\stackrel{h\rightarrow 0_+
}{\longrightarrow} \int_{\mathbb{R^d} }\nu \left( x ,\xi, t \right)  d\xi \ \text{in }%
\mathcal{M}^{+}\left( \mathbb{R}^d_x\right)\text{w--}\ast .
\end{equation*}%
Under our assumption \eqref{a1} on $V$, this can be used to prove that (see \cite{Jin:2017bh} for more details):
\begin{eqnarray*}
F^{h }\left( y,t\right) \stackrel{h\rightarrow 0_+
}{\longrightarrow}-\iint_{\R^{2d}} \nabla _{y}V\left(
x,y\right) \nu ( x, \xi, t) \, dx \, d\xi=: F^{0}\left( y,t\right) ,
\end{eqnarray*}
uniformly on compact intervals in $y$ and $t$.

Similarly, one can pass to the limit $h \to 0_+$ in the equation for $\mu^h$ to find that there exists a limiting measure
$\mu^0 \equiv \mu \in \mathcal M^+(\R^n_y\times \R^n_\eta)$ which consequently solves (in the sense of distributions):
\begin{equation}\label{liou liou 1}
\partial _{t}\mu+\eta \cdot \nabla _{y}\mu+F^{0}\left( y,t\right)
\cdot \nabla _{\eta }\mu =0.
\end{equation}
Moreover, one can prove that
\begin{eqnarray*}
\Upsilon ^{h }\left( x,t\right) \stackrel{h\rightarrow 0_+
}{\longrightarrow}\iint_{\mathbb{R}^{2n}}V\left( x,y\right) \mu\left(
y,\eta ,t\right) \, dy \, d \eta=: \Upsilon ^{0}\left( x,t\right) .
\end{eqnarray*}
In view of the definition \eqref{theta}, one also finds that
\[
\Theta ^{h }\left[ \Upsilon ^{h }\right] w^{h }\stackrel{h\rightarrow 0_+
}{\longrightarrow} -\nabla _{x}\Upsilon ^{0}\left(
x,t\right) \cdot \nabla _{\xi }\nu
\]
and thus, the Wigner measure associated to $\psi^h$ satisfies the following Liouville equation (in the sense of distributions):
\begin{equation}\label{liou liou 2}
\partial _{t}\nu+\xi \cdot \nabla _{x}\nu-\nabla _{x}\Upsilon ^{0}\left(
x,t\right) \cdot \nabla _{\xi }\nu=0.
\end{equation}
In summary, one finds a system of two coupled Liouville equations \eqref{liou liou 1}--\eqref{liou liou 2} in the classical limit
(we refer to \cite{Jin:2017bh} for a rigorous proof and further details).


\section{The spatial meshing strategy}\label{sec:mesh}

\subsection{The semi-discretized SLE system and its energy} The analysis in this section will focus on the spatial meshing strategy. In order to show that
it is possible to use a grid with $\Delta y, \Delta \eta \sim O(1)$, and thus, independent of $h$, we will
consider a semi-discretized version of the SLE system \eqref{SL-1} in one spatial dimension $d=n=1$ where
the Liouville is discretized using an upwind scheme:
\begin{equation}\label{quad-sle}
\left\{
\begin{split}
& ih \partial _{t}\psi ^{h }=-\frac{h ^{2}}{2}\partial
_{xx}\psi ^{h }+\Upsilon _{d}^{h }\left( x,t\right) \psi ^{h
},\quad \psi ^{h }(0,x)=\psi _{\rm in}^{h }(x) ,   \\
& \partial _{t}\mu ^{h }+\eta D_{y}\mu ^{h }+F^{h }\left(
y,t\right) D_{\eta }\mu ^{h }=0,\quad  \mu ^{h }(0,y, \eta)=\mu
_{\rm in}^{h } (y,\eta).
\end{split}
\right.
\end{equation}
Here, $\Upsilon _{d}^{h }\left( x,t\right) $ stands for the trapezoidal
quadrature approximation of $\Upsilon ^{h }\left( x,t\right) $, as before, whereas
$F^{h }\left( y,t\right) $ includes
exact derivative of the known function $V\left( x,y\right) $. We shall refer to \eqref{quad-sle} as the semi-discretized SLE system (s-SLE)
and show that it yields the ``correct" classical limit, i.e., the semi-discretized version of \eqref{liou liou 1}--\eqref{liou liou 2}.

Before doing so, we will need to prove an a-priori estimate and the energy associated to \eqref{quad-sle}. To this end, we
define the semi-discrete energy as
\begin{eqnarray*}
E_{d}\left( t\right) :=\int_{\mathbb{R}}\frac{h ^{2}}{2}\left\vert
\partial _{x}\psi ^{h }(x,t)\right\vert ^{2}dx+\int_{\mathbb{R}}\Upsilon
_{d}^{h }\left( x,t\right) \left\vert \psi ^{h }\right\vert
^{2}dx+\sum\limits_{j=0}^{J-1}\sum\limits_{k=0}^{K-1}\frac{\eta _{k}^{2}}{2}%
\mu _{jk}^{h }\Delta y\Delta \eta.
\end{eqnarray*}%
Here, and in the following, because of the periodicity of $\mu $, we shall use a cyclic index for $\mu _{jk}$,
such that $\mu _{jk}=\mu _{j+J,k}=\mu _{j-J,k}=\mu _{j,k+K}=\mu _{j,k-K}$.

\begin{theorem}\label{thm:energy}
\label{energy estimate} Under the assumptions \eqref{a1} and \eqref{a2}, the energy $E_{d}(t)$ is bounded by a constant independent of $h $ for
all $t\geq 0$.
\end{theorem}

\begin{proof} We start by showing that the initial energy is bounded. This is easily seen from
\begin{equation*}
E_{d}\left( 0\right) =\int_{\mathbb{R}}\frac{h ^{2}}{2}\left\vert
\partial _{x}\psi _{\rm in}^{h }\right\vert ^{2}dx+\int_{\mathbb{R}%
}\Upsilon _{d}^{h }\left( x,0\right) \left\vert \psi _{\rm in}^{h
}\right\vert ^{2}dx+\sum\limits_{j=0}^{J-1}\sum\limits_{k=0}^{K-1}\left( \mu
_{\rm in}\right) _{jk}\Delta y\Delta \eta ,
\end{equation*}
where the first two integrals are clearly bounded by assumptions \eqref{a1} and \eqref{a2} and the last term is just a quadrature
approximation of $\iint \mu _{\rm in} \,dy \, d\eta =1$, and hence bounded.

Next, we compute the time-derivative of $E_d$ as
\begin{equation*}
\frac{d}{dt}E_{d} =(\text{\uppercase\expandafter{%
\romannumeral1}})+(\text{\uppercase\expandafter{\romannumeral2}})+(\text{%
\uppercase\expandafter{\romannumeral3}})+\text{(\uppercase%
\expandafter{\romannumeral4}),}\ \
\end{equation*}%
where
\begin{eqnarray*}
(\text{\uppercase\expandafter{\romannumeral1}}) &:= &\int_{\mathbb{R}
}\frac{h ^{2}}{2}\left( \partial _{x}\partial _{t}\bar{\psi}^{h
}\cdot \partial _{x}\psi ^{h }+\partial _{x}\bar{\psi}^{h }\cdot
\partial _{x}\partial _{t}\psi ^{h }\right)dx, \\
(\text{\uppercase\expandafter{\romannumeral2}}) &:= &\int_{\mathbb{R}
}\Upsilon _{d}^{h }\left( x,t\right) \left (\partial _{t}\bar{\psi}%
^{h }\psi ^{h }+\bar{\psi}^{h }\partial _{t}\psi ^{h }%
\right )dx, \\
(\text{\uppercase\expandafter{\romannumeral3}}) &:= &\int_{\mathbb{R}
}\partial _{t}\Upsilon _{d}^{h }\left( x,t\right) \left\vert \psi
^{h }\right\vert ^{2}dx, \\
(\text{\uppercase\expandafter{\romannumeral4}}) &:=
&\sum\limits_{j=0}^{J-1}\sum\limits_{k=0}^{K-1}\frac{\eta _{k}^{2}}{2}\left(
\partial _{t}\mu ^{h }\right) _{jk}\Delta y\Delta \eta .
\end{eqnarray*}
First, a straightforward calculation shows $($\uppercase\expandafter{\romannumeral1}%
$)+($\uppercase\expandafter{\romannumeral2}$)=0$, since
\begin{eqnarray*}
(\text{\uppercase\expandafter{\romannumeral1}})+(\text{\uppercase%
\expandafter{\romannumeral2}}) &=&-\int_{\mathbb{R} }\frac{h ^{2}}{2}%
\left( \partial _{t}\bar{\psi}^{h }\partial _{xx}\psi ^{h
}+\partial _{xx}\bar{\psi}^{h }\partial _{t}\psi ^{h }\right)
dx+\int_{\mathbb{R} }\Upsilon _{d}^{h }\left ( \partial _{t}\bar{\psi}%
^{h }\psi ^{h }+\bar{\psi}^{h }\partial _{t}\psi ^{h }%
\right) dx \\
&=&\int_{\mathbb{R} }\partial _{t}\bar{\psi}^{h }\left( -\frac{h
^{2}}{2}\partial _{xx}\psi ^{h }+\Upsilon _{d}^{h }\psi ^{h
}\right) +\left( -\frac{h ^{2}}{2}\partial _{xx}\bar{\psi}^{h
}+\Upsilon _{d}^{h }\bar{\psi}^{h }\right) \partial _{t}\psi
^{h }dx \\
&=&\int_{\mathbb{R} }\partial _{t}\bar{\psi}^{h }\left( ih
\partial _{t}\psi ^{h }\right) +\left( -ih \partial _{t}\bar{\psi}%
^{h }\right) \partial _{t}\psi ^{h }dx =0.
\end{eqnarray*}
For simplicity we will, from now on, denote
\begin{equation*} \label{Gj}
G_{j}^{h }(t)\equiv G^{h }\left(t, x,y_{j}\right) =\int_{\mathbb{R}%
}V\left( x,y_{j}\right) \left\vert \psi ^{h }(t,x)\right\vert ^{2}\,dx\geq 0,
\end{equation*}
as well as
\begin{equation*}
F_{j}^{h} (t)\equiv F^{h }\left( y_{j},t\right) =-\int_{\mathbb{R}%
}\partial _{y}V\left( x,y_{j}\right) \left\vert \psi ^{h }\right\vert
^{2}dx.  \label{Fj}
\end{equation*}%
A key observation is that $G^{h }$ is in
fact Lipschitz with a Lipschitz constant $L>0$ independent of $h $, since
\begin{align*}
\left\vert G_{j+1}^{h }-G_{j}^{h }\right\vert =\left\vert \int_{%
\mathbb{R}}\left[ V\left( x,y_{j+1}\right) -V\left( x,y_{j}\right) \right]
\left\vert \psi ^{h }\right\vert ^{2}dx\right\vert \leq \int_{\mathbb{R}}\left\vert \partial _{y}V\left( x,\xi \right)
\right\vert \left\vert y_{j+1}-y_{j}\right\vert \left\vert \psi ^{h
}\right\vert ^{2}dx ,
\end{align*}
for some $\xi \in \left( y_{j},y_{j+1}\right) $. Thus
\begin{align*}
\left\vert G_{j+1}^{h }-G_{j}^{h }\right\vert
\leq \left\Vert \partial _{y}V\left( x,\xi \right) \right\Vert _{L^{\infty
}}\left\Vert \psi ^{h }\right\Vert _{L^{2}}^{2}\left\vert
y_{j+1}-y_{j}\right\vert =\left\Vert \partial _{y}V\left( x,\xi \right) \right\Vert _{L^{\infty
}}\Delta y =: L\Delta y,
\end{align*}
since $\left\Vert \psi ^{h
}\right\Vert _{L^{2}}=\left\Vert \psi _{\rm in}^{h }\right\Vert _{L^{2}}=1$, in view of mass conservation established in Theorem \ref{mass cons}. In addition, we have that $F^h_j$ is uniformly bounded, i.e.
\begin{equation}
\left\vert F_{j}^{h}\right\vert =\left\vert \int_{\mathbb{R}}\partial
_{y}V\left( x,y_{j}\right) \left\vert \psi ^{h }\right\vert
^{2}dx\right\vert \leq \left\Vert \partial _{y}V\left( x,\xi \right)
\right\Vert _{L^{\infty }}\left\Vert \psi ^{h }\right\Vert
_{L^{2}}^{2}=L.  \label{f bdd}
\end{equation}%

Coming back to $($\uppercase\expandafter{\romannumeral3}$)$, we first note
\begin{eqnarray*}
\partial _{t}\Upsilon _{d}^{h }\left( x,t\right)
&=&\sum\limits_{j=0}^{J-1}\sum\limits_{k=0}^{K-1}V\left( x,y_{j}\right)
\left( \partial _{t}\mu ^{h }\right) _{jk}\Delta y\Delta \eta \\
&=&-\sum\limits_{j=0}^{J-1}\sum\limits_{k=0}^{K-1}V\left( x,y_{j}\right)
\left[ \eta _{k}\left( D_{y}\mu ^{h }\right) _{jk}+F_{j}^{h
}\left( D_{\eta }\mu ^{h }\right) _{jk}\right] \Delta y\Delta \eta \\
&=&-\sum\limits_{j=0}^{J-1}\sum\limits_{k=0}^{K-1}V\left( x,y_{j}\right)
\eta _{k}\left( D_{y}\mu ^{h }\right) _{jk}\Delta y\Delta \eta,
\end{eqnarray*}%
where the last equality follows from the fact that we have a telescoping series in $k$ with zero boundary conditions.
Recalling (\ref{Gj}), $($\uppercase\expandafter{\romannumeral3}$)$ can be written as
\begin{eqnarray*}
(\text{\uppercase\expandafter{\romannumeral3}})
&=&-\sum\limits_{j=0}^{J-1}\sum\limits_{k=0}^{K-1}G_{j}^{h }\eta
_{k}\left( D_{y}\mu ^{h }\right) _{jk}\Delta y\Delta \eta \\
&=&-\sum\limits_{j=0}^{J-1}\sum\limits_{k=0}^{K-1}G_{j}^{h }\frac{\eta
_{k}+\left\vert \eta _{k}\right\vert }{2}\left( \mu _{jk}^{h }-\mu
_{j-1,k}^{h }\right) \Delta \eta
-\sum\limits_{j=0}^{J-1}\sum\limits_{k=0}^{K-1}G_{j}^{h }\frac{\eta
_{k}-\left\vert \eta _{k}\right\vert }{2}\left( \mu _{j+1,k}^{h }-\mu
_{j,k}^{h }\right) \Delta \eta \\
&=&\sum\limits_{j=0}^{J-1}\sum\limits_{k=0}^{K-1}\frac{\eta _{k}+\left\vert
\eta _{k}\right\vert }{2}\mu _{jk}^{h }\left( G_{j+1}^{h
}-G_{j}^{h }\right) \Delta \eta
+\sum\limits_{j=0}^{J-1}\sum\limits_{k=0}^{K-1}\frac{\left\vert \eta
_{k}\right\vert -\eta _{k}}{2}\mu _{jk}^{h }\left( G_{j-1}^{h
}-G_{j}^{h }\right) \Delta \eta ,
\end{eqnarray*}
where summation by parts is used in the last equality.
In view of the Lipschitz property above, it is then straightforward to estimate $($\uppercase\expandafter{\romannumeral3}$)$ via
\begin{eqnarray*}
(\text{\uppercase\expandafter{\romannumeral3}})
\leq L\sum\limits_{j=0}^{J-1}\sum\limits_{k=0}^{K-1}\left\vert \eta
_{k}\right\vert \mu _{jk}^{h }\Delta y\Delta \eta .
\end{eqnarray*}

Similarly, one proves that
\begin{eqnarray*}
(\text{\uppercase\expandafter{\romannumeral4}})
&=&-\sum\limits_{j=0}^{J-1}\sum\limits_{k=0}^{K-1}\frac{\eta _{k}^{2}}{2}%
F_{j}^{h }\left( D_{\eta }\mu ^{h }\right) _{jk}\Delta y\Delta \eta
\\
&=&-\sum\limits_{j=0}^{J-1}\sum\limits_{k=0}^{K-1}\frac{\eta _{k}^{2}}{2}%
\frac{F_{j}^{h }+\left\vert F_{j}^{h }\right\vert }{2}\left( \mu
_{jk}^{h }-\mu _{j,k-1}^{h }\right) \Delta
y-\sum\limits_{j=0}^{J-1}\sum\limits_{k=0}^{K-1}\frac{\eta _{k}^{2}}{2}\frac{%
F_{j}^{h }-\left\vert F_{j}^{h }\right\vert }{2}\left( \mu
_{j,k+1}^{h }-\mu _{jk}^{h }\right) \Delta y
\\
&=&\sum\limits_{j=0}^{J-1}\sum\limits_{k=0}^{K-1}\frac{\eta _{k+1}^{2}-\eta
_{k}^{2}}{2}\frac{F_{j}^{h }+\left\vert F_{j}^{h }\right\vert }{2}%
\mu _{jk}^{h }\Delta y+\sum\limits_{j=0}^{J-1}\sum\limits_{k=0}^{K-1}%
\frac{\eta _{k}^{2}-\eta _{k-1}^{2}}{2}\frac{F_{j}^{h }-\left\vert
F_{j}^{h }\right\vert }{2}\mu _{jk}^{h }\Delta y
\\
&=&\sum\limits_{j=0}^{J-1}\sum\limits_{k=0}^{K-1}\left( \eta _{k}+\frac{%
\Delta \eta }{2}\right) \Delta \eta \frac{F_{j}^{h }+\left\vert
F_{j}^{h }\right\vert }{2}\mu _{jk}^{h }\Delta
y+\sum\limits_{j=0}^{J-1}\sum\limits_{k=0}^{K-1}\left( \eta _{k}-\frac{%
\Delta \eta }{2}\right) \Delta \eta \frac{F_{j}^{h }-\left\vert
F_{j}^{h }\right\vert }{2}\mu _{jk}^{h }\Delta y ,
\end{eqnarray*}
where summation by parts is used as before. Combining the coefficients of  $F_j^h$ and $\left\vert F_j^h \right \vert$ respectively, this is equal to
\begin{eqnarray*}
(\text{\uppercase\expandafter{\romannumeral4}})
&=&\sum\limits_{j=0}^{J-1}\sum\limits_{k=0}^{K-1}\eta _{k}F_{j}^{h }\mu
_{jk}^{h }\Delta y\Delta \eta
+\sum\limits_{j=0}^{J-1}\sum\limits_{k=0}^{K-1}\frac{\Delta \eta }{2}%
\left\vert F_{j}^{h }\right\vert \mu _{jk}^{h }\Delta y\Delta \eta
\\
&\leq &L\sum\limits_{j=0}^{J-1}\sum\limits_{k=0}^{K-1}\left\vert \eta
_{k}\right\vert \mu _{jk}^{h }\Delta y\Delta \eta
+L\sum\limits_{j=0}^{J-1}\sum\limits_{k=0}^{K-1}\frac{\Delta \eta }{2}\mu
_{jk}^{h }\Delta y\Delta \eta .
\end{eqnarray*}

In summary, we thus find
\begin{align*}
&\frac{d}{dt} E_{d} \leq 2L\sum\limits_{j=0}^{J-1}\sum\limits_{k=0}^{K-1}\left\vert \eta
_{k}\right\vert \mu _{jk}^{h }\Delta y\Delta \eta
+L\sum\limits_{j=0}^{J-1}\sum\limits_{k=0}^{K-1}\frac{\Delta \eta }{2}\mu
_{jk}^{h }\Delta y\Delta \eta  \notag \\
& \leq 2L\left( \sum\limits_{j=0}^{J-1}\sum\limits_{k=0}^{K-1}\eta
_{k}^{2}\mu _{jk}^{h }\Delta y\Delta \eta \right) ^{\frac{1}{2}}\left(
\sum\limits_{j=0}^{J-1}\sum\limits_{k=0}^{K-1}\mu _{jk}^{h }\Delta
y\Delta \eta \right) ^{\frac{1}{2}}+L\sum\limits_{j=0}^{J-1}\sum%
\limits_{k=0}^{K-1}\frac{\Delta \eta }{2}\mu _{jk}^{h }\Delta y\Delta \eta .
\end{align*}
Using the fact that
$$\sum\limits_{j=0}^{J-1}\sum\limits_{k=0}^{K-1}%
\mu _{jk}^{h }\Delta y\Delta \eta =C$$ is a conserved quantity with respect to time, we consequently find the following estimate
\begin{eqnarray*}
\frac{d}{dt} E_{d}\left( t\right) &\leq &2^{\frac{3}{2}}LC^{\frac{1}{2}}\sqrt{E_{d}\left( t\right) }+\frac{LC}{%
2}\Delta \eta \\
&\leq &E_{d}\left( t\right) +2L^{2}C+\frac{LC}{2}\Delta \eta \equiv
E_{d}\left( t\right) +C_{1}.
\end{eqnarray*}
By Gronwall's inequality, this yields
\begin{equation*}
E_{d}\left( t\right) \leq \left( C_{1}+E_d(0)\right) e^{t}-C_{1}\text{,}
\end{equation*}%
which gives the desired bound independent of $h $.
\end{proof}

\begin{remark} It is easy to find a sharper bound of the energy by considering times $t\leq e$ and $t> e$, respectively, but the estimate above is sufficient for our purposes.
\end{remark}


\subsection{The classical limit of the s-SLE system}

In this section, we shall perform the limit $h \rightarrow 0^{+}$ of the s-SLE system \eqref{quad-sle}.
By proving that it converges, as $h \to 0_+$, to the semi-discretized version of the coupled Liouville-system  \eqref{liou liou 1}--\eqref{liou liou 2},
we infer that it is possible to choose a spatial meshing strategy such that $\Delta y, \Delta \eta \sim \mathcal O(1)$.

To this end, we first note that the a-priori bounds on the mass and energy obtained in Theorems \ref{mass cons} and \ref{thm:energy},
together with our assumptions on $V\geq 0$ imply that the solution $\psi ^{h }$ of \eqref{quad-sle} is
uniformly bounded in $L^{2}\left(\mathbb{R}\right) $, and $h -$oscillatory. Thus, there exists an associated
Wigner measure $\nu(,\cdot, \cdot, t)\in \mathcal{M}^{+}( \mathbb{R}%
_{x}\times \mathbb{R}_{\xi })$ and we directly infer that
\begin{eqnarray*}
F^{h }\left( y,t\right) \stackrel{h\rightarrow 0_+
}{\longrightarrow}-\iint_{\R^{2}} \partial _{y}V\left(
x,y\right) \nu (t, x, \xi) \, dx \, d\xi=: F^{0}\left( y,t\right) ,
\end{eqnarray*}
by the same arguments as in \cite{Jin:2017bh} (recall that $x,t$ are taken to be continuous in \eqref{quad-sle}). In the following we shall use the short-hand notation
$F^h_j(t)$ and $F_j^0(t)$, respectively, as given in \eqref{Fj}.

\subsection{Convergence of $\mu ^{h }$}

Next, we turn to the solution $\mu ^{h }$ within \eqref{quad-sle}, which we recall to be discretized via an upwind scheme. We shall prove the following result about its limiting behavior as $h\to 0_+$. In the following, $C_{\rm b}(\mathbb R)$ denotes the space of continuous and bounded functions on $\mathbb R$.

\begin{proposition}
Let$\ \mu _{jk}^{h }\left( t\right) \in C_{\rm b}\left( \mathbb{R}%
_{t}\right) $ be a solution of
\begin{equation*}
\frac{d}{dt}\mu _{jk}^{h }\left( t\right) =-\eta _{k}D_{y}\mu
_{jk}^{h }-F_{j}^{h }\left( t\right) D_{\eta }\mu _{jk}^{h },
\end{equation*}%
and $\mu _{jk}\left( t\right) \in C_{\rm b}\left( \mathbb{R}_{t}\right) $ be
a solution of%
\begin{equation*}
\frac{d}{dt}\mu _{jk}\left( t\right) =-\eta _{k}D_{y}\mu _{jk}\left(
t\right) -F_{j}^{0}\left( t\right) D_{\eta }\mu _{jk}\left( t\right) ,
\end{equation*}%
where $j=0,\cdots ,J-1$ and $k=0,\cdots ,K-1$, such that initially $\mu
_{jk}^{h }\left( 0\right) =\mu _{jk}\left( 0\right) $. Then for any
given $T>0$, %
\begin{equation*}
\mu _{jk}^{h }\stackrel{h\rightarrow 0_+
}{\longrightarrow} \mu _{jk}^0\equiv \mu_{jk}\text{, as }h \rightarrow
0_{+}\ \text{in }L^{\infty } \left[ 0,T\right] \text{,}
\end{equation*}%
up to the extraction of subsequences.
\end{proposition}

\begin{proof}
Denote the difference between $\mu _{jk}^{h }$ and its limit by
\[
e_{jk}^{h}\left( t\right) =\mu _{jk}^{h }(t)-\mu _{jk}(t),
\]
which solves the following system of equations:
\begin{eqnarray*}
\frac{d}{dt}e_{jk}^{h }\left( t\right) &=&-\eta _{k}D_{y}e_{jk}^{h
}-F_{j}^{h }\left( t\right) D_{\eta }e_{jk}^{h }+F_{j}^{0}\left(
t\right) D_{\eta }\mu _{jk}-F_{j}^{h }\left( t\right) D_{\eta }\mu
_{jk} \\
&=&-\frac{1}{2 \Delta y}(\eta _{k}+\left\vert \eta _{k}\right\vert)\left(
e_{jk}^{h }-e_{j-1,k}^{h }\right) -\frac{1}{2 \Delta y}(\eta _{k}-\left\vert \eta
_{k}\right\vert )\left( e_{j+1,k}^{h }-e_{jk}^{h }\right) \\
&&-\frac{1}{2 \Delta \eta}(F_{j}^{h }\left( t\right) +\left\vert F_{j}^{h }\left(
t\right) \right\vert )\left( e_{jk}^{h }-e_{j,k-1}^{h }\right) -%
\frac{1}{2 \Delta \eta}(F_{j}^{h }\left( t\right) -\left\vert F_{j}^{h }\left(
t\right) \right\vert )\left( e_{j,k+1}^{h }-e_{jk}^{h }\right)
\\
&&+\frac{1}{2 \Delta \eta}(F_{j}^{0}\left( t\right) +\left\vert F_{j}^{0}\left( t\right)
\right\vert -F_{j}^{h }\left( t\right) -\left\vert F_{j}^{h
}\left( t\right) \right\vert )\left( \mu _{jk}-\mu _{j,k-1}\right)
\\
&&+\frac{1}{2 \Delta \eta}(F_{j}^{0}\left( t\right) -\left\vert F_{j}^{0}\left( t\right)
\right\vert -F_{j}^{h }\left( t\right) +\left\vert F_{j}^{h
}\left( t\right) \right\vert )\left( \mu _{j,k+1}-\mu _{jk}\right),
\end{eqnarray*}
subject to initial data $e_{jk}^{h}\left( 0 \right)=0$, since $\mu _{jk}^{h }(0)=\mu _{jk}(0)$.
For simplicity we shall write the system above in vector form, i.e.,
\begin{equation}\label{ode}
\frac{d}{dt}E^{h} (t)=A^{h }\left( t\right) E^{h
}\left( t\right) +b^{h }\left( t\right) ,\ E^{h }\left( 0\right)
=0.
\end{equation}%
Here $E^{h }\left( t\right) $ and $b^{h }\left( t\right) $ are
both $JK$-dimensional vectors, $A^{h }\left( t\right) $
is a continuous $JK \times JK$ matrix-valued function of $t\in \left( 0,T%
\right] $, and
\begin{eqnarray*}
b_{\left( j-1\right) K+k}^{h }\left( t\right) &=&\frac{1}{2 }(F_{j}^{0}\left(
t\right) +\left\vert F_{j}^{0}\left( t\right) \right\vert -F_{j}^{h
}\left( t\right) -\left\vert F_{j}^{h }\left( t\right) \right\vert )%
\frac{ \mu _{jk}-\mu _{j,k-1}}{\Delta y} \\
&&+\frac{1}{2}(F_{j}^{0}\left( t\right) -\left\vert F_{j}^{0}\left( t\right)
\right\vert -F_{j}^{h }\left( t\right) +\left\vert F_{j}^{h
}\left( t\right) \right\vert )\frac{ \mu _{j,k+1}-\mu
_{jk}}{\Delta \eta} ,
\end{eqnarray*}%
for $j=1,\ldots ,J$, and $k=1,\ldots ,K$. Clearly, $\left\Vert A^{h }\left( t\right) \right\Vert
_{\infty }\leq C$, where $C$ is a constant independent of $h $, due to
the fact that\ $\left\vert F_{j}^{h }\left( t\right) \right\vert $ and $%
\left\vert F_{j}^{0}\left( t\right) \right\vert $ are both bounded by some
constants independent of $h $, see the proof of Theorem \ref{thm:energy}. Classical ODE theory then implies (see, e.g., \cite{taylor1996partial}), that
there is a matrix-valued function $%
S^{h }\left( t,s\right) $ such that the solution of \eqref{ode} is given by Duhamel's principle:
\begin{eqnarray*}
E^{h }\left( t\right) &=&S^{h }\left( t,0\right) E^{h }\left(
0\right) +\int_{0}^{t}S^{h }\left( t,s\right) b^{h }\left(
s\right) ds \\
&=&\int_{0}^{t}S^{h }\left( t,s\right) b^{h }\left( s\right) ds.
\end{eqnarray*}%
Moreover, there exists a bound on the propagator $S$ of the
\begin{equation}\label{c3}
\left\Vert S^{h }\left( t,s\right) \right\Vert _{\infty }\leq C_{3},
\end{equation}
where $C_{3}$ is a constant independent of $h $.

Next, we recall that $F_{j}^{h }\left( t\right) =F^{h }\left( y_{j},t\right)$ is uniformly bounded, by equation (\ref{f bdd}), and
\begin{equation*}
F_{j}^{h }\left( t\right)  \stackrel{h\rightarrow 0_+
}{\longrightarrow}
F_{j}^{0}\left( t\right) =F^{0}\left( y_{j},t\right) \text{, as }h
\rightarrow 0^{+},
\end{equation*}%
pointwise (up to the extraction of subsequences). In addition, $F_{j}^{h }\left(
t\right) $ is easily seen to be equi-continuous in time, by the same type of argument as in \cite{Jin:2017bh}. Namely, by using Schr\"odinger's equation, one finds
\begin{eqnarray*}
\left\vert \partial _{t}F_{j}^{h }\left( t\right) \right\vert
&=&\left\vert \int_{\mathbb{R} }\partial _{y}V\left( x,y_{j}\right) \left(
\partial _{t}\bar{\psi}^{h }\psi ^{h }+\bar{\psi}^{h
}\partial _{t}\psi ^{h }\right) dx\right\vert
\\
&=&\left\vert \int_{\mathbb{R} }\frac{ih }{2}\partial _{y}V\left(
x,y_{j}\right) \left( \bar{\psi}^{h }\partial _{xx}\psi ^{h
}-\partial _{xx}\bar{\psi}^{h }\psi ^{h }\right) dx\right\vert .
\end{eqnarray*}
Integrating by parts, it reads
\begin{eqnarray*}
\left\vert \partial _{t}F_{j}^{h }\left( t\right) \right\vert
&=&\frac{h }{2}\left\vert \int_{\mathbb{R} }\partial _{x}\left(
\partial _{y}V\left( x,y_{j}\right) \bar{\psi}^{h }\right) \partial
_{x}\psi ^{h }-\partial _{x}\bar{\psi}^{h }\partial _{x}\left(
\partial _{y}V\left( x,y_{j}\right) \psi ^{h }\right) dx\right\vert \\
&=&\frac{h }{2}\left\vert \int_{\mathbb{R} }\partial _{xy}V\left(
x,y_{j}\right) \left(\bar{\psi}^{h }\partial _{x}\psi ^{h
}-\partial _{x}\bar{\psi}^{h }\psi \right )dx\right\vert \\
&\leq &h \left\Vert \partial _{xy}V\left( x,y\right) \right\Vert
_{L^{\infty }\left( \mathbb{R}^{2}\right) }\| \psi^h \|_{L^2(\R)} \| \partial_x \psi^h \|_{L^2(\R)}\leq C\left( t\right) \text{,}
\end{eqnarray*}%
where the last inequality follows from the $h-$oscillatory nature of $\psi^h$.

This consequently implies that
\begin{equation*}
b^{h }\left( t\right) \stackrel{h\rightarrow 0_+
}{\longrightarrow} 0\text{, as }h \rightarrow 0_{+}%
\text{,}
\end{equation*}%
locally uniformly in $t$, up to extraction of some subsequence, which in turn yields
convergence of
$E^{h }\left( t\right) $ itself, as can be seen by considering its $m$-th component, for $m=1,\dots ,JK$:
\begin{eqnarray*}
\left\vert E_{m}^{h }\left( t\right) \right\vert &=&\left\vert
\int_{0}^{t}\sum\limits_{n=1}^{JK}S_{mn}^{h }\left( t,s\right)
b_{n}^{h }\left( s\right) ds\right\vert \\
&\leq &C_{3}\sum\limits_{n=1}^{JK}\int_{0}^{t}\left\vert b_{n}^{h
}\left( s\right) \right\vert ds\rightarrow 0\text{, as }h \rightarrow
0_{+},
\end{eqnarray*}
where we have used \eqref{c3}.
\end{proof}

\subsection{Equation for $\nu $ and the main result}

With the convergence theorem of $\mu^h$ in hand, we can now state the following result, which represents the final step in our analysis.

\begin{proposition}
Assume \eqref{a1} and \eqref{a2} and let $\Theta [ \Upsilon _{d}^{h }]$ be the pseudo-differential operator defined in \eqref{theta} applied to the trapezoidal
quadrature approximation of $\Upsilon ^{h }\left( x,t\right) $. Then it holds
\begin{equation*}
\Theta [ \Upsilon _{d}^{h }] w^{h }\left( x,\xi
,t\right) \stackrel{h\rightarrow 0_+
}{\longrightarrow} -\partial _{x}\Upsilon _{d}^{0}\left( x,t\right)
\partial _{\xi }\nu \left( x,\xi ,t\right) \text{\ \ in \thinspace }%
L^{\infty }\left([0,T];\mathcal{A}^{\prime }\left( \mathbb{R}%
_{x}\times \mathbb{R}_{\xi }\right) \rm w\text{--}\ast\right) \text{,}
\end{equation*}
where
\begin{equation*}
\Upsilon _{d}^{0}\left( x,t\right)
=\sum\limits_{j=0}^{J-1}\sum\limits_{k=0}^{K-1}V\left( x,y_{j}\right) \mu
_{jk}^{0}\Delta y\Delta \eta .
\end{equation*}
\end{proposition}

The proof of this proposition follows from the same arguments as given in the proof of Lemma 4.5 in
\cite{Jin:2017bh}, and we therefore omit it here.

We are now in the position to state the main result of this section.

\begin{theorem}
Let Assumption \eqref{a1} and \eqref{a2} hold. Then, for any $T>0$, the solution of semi-discretized SLE system \eqref{quad-sle} satisfies, up to extraction of sub-sequences,
\begin{equation*}
w^{h }[ \psi ^{h }]  \stackrel{h\rightarrow 0_+
}{\longrightarrow}
\nu\text{ in }L^\infty([0,T];\mathcal{A}^{\prime }\left( \mathbb{R}_{x}\times
\mathbb{R}_{\xi }\right))\, {\rm w}\text{--}\ast ,\quad \mu _{jk}^{h } \stackrel{h\rightarrow 0_+
}{\longrightarrow} \mu
_{jk}^{0}\text{ in }L^{\infty } \left[ 0,T\right] ,
\end{equation*}
where $j=0,\cdots ,J-1$ and $k=0,\cdots ,K-1$. In addition, $\nu $ and $\mu _{jk}$
solve the semi-discretized Liouville-system
\begin{equation*}
\left\{
\begin{split}
&\, \partial _{t}\nu +\xi \partial _{x}\nu -\partial _{x}\Upsilon _{d}^{0}\left(
x,t\right) \partial _{\xi }\nu =0, \\
&\, \frac{d}{dt}\mu _{jk}^{0}+ \eta _{k}D_{y}\mu _{jk}^{0} + F_{j}^{0}D_{\eta }\mu _{jk}^{0}=0.
\end{split}
\right.
\end{equation*}
\end{theorem}
\begin{remark}
Numerical experiments show that the same type of behavior is true not only for mixed spectral-finite difference schemes, but also
purely spectral schemes, see \cite{Jin:2017bh}. Our proof, however, only works for the former case due to the required positivity of the energy.
\end{remark}


\section{Time-discretization}\label{sec:time}

We finally turn to the time-discretization of our splitting scheme (in one dimension  $d=n=1$) as given by \eqref{ts1}, \eqref{ts2}.
In this section we want to show that it is asymptotic preserving in the sense that in the limit $h\to 0_+$, it yields
the corresponding time-splitting scheme of \eqref{liou liou 1}--\eqref{liou liou 2}, i.e.,
\begin{equation} \label{macro1}
\left\{
\begin{split}
&\partial _{t}\nu +\xi \partial _{x}\nu =0, \\
&\frac{d}{dt}\mu_{jk}+\eta _{k}\left( D_{y}\mu \right) _{jk}+F_{j}^{0}\left( D_{\eta }\mu\right) _{jk}=0,
\end{split}
\right.
\end{equation}
and
\begin{equation} \label{macro2}
\left\{
\begin{split}
&\partial _{t}\nu -\partial _{x}\Upsilon _{d}^{0}\left( x,t\right) \partial
_{\xi }\nu =0, \\
& \frac{d}{dt}\mu _{jk}=0.
\end{split}
\right.
\end{equation}
In turn, this shows that $\Delta t\sim \mathcal O(1)$ can be chosen independent of the small parameter $h$. To this end, it suffices to show that in our time-spitting method $\psi ^{h }$ is $h-$oscillatory, i.e.
\begin{equation*}
\underset{0<h \leq 1}{\sup }\left\Vert h \partial _{x}\psi
^{h }\right\Vert _{L_{x}^{2}}\leq C,
\end{equation*}%
where the constant $C$ depends only on the final time $T$. Then, following the arguments given in the previous section, one has convergence of the forcing term $F^h_j$ as $h\to 0$. In turn, 
this yields convergence of our numerical scheme towards the corresponding scheme of the limiting equation, as stated in \eqref{macro1} and \eqref{macro2}. 

We consequently consider the splitting scheme \eqref{ts1}, \eqref{ts2} and recall
that in both splitting steps, the first equation, i.e., the quantum
part is solved exactly in time. A straightforward calculation then shows that $\left\Vert h \partial _{x}\psi ^{h }\right\Vert^2
_{L_{x}^{2}}$ is conserved in the first splitting step \eqref{ts1}, i.e.
\begin{eqnarray*}
&&\frac{d}{dt}\int h ^{2}\left\vert \partial _{x}\psi ^{h
}\right\vert ^{2}dx=h ^{2}\int \partial _{tx}\overline{\psi ^{h }}%
\partial _{x}\psi ^{h }+\partial _{x}\overline{\psi ^{h }}\partial
_{tx}\psi ^{h }dx \\
&=&-h ^{2}\int \partial _{t}\overline{\psi ^{h }}\partial
_{xx}\psi ^{h }+\partial _{xx}\overline{\psi ^{h }}\partial
_{t}\psi ^{h }dx=-h ^{2}\int -\frac{ih }{2}\partial _{xx}%
\overline{\psi ^{h }}\partial _{xx}\psi ^{h }+\partial _{xx}%
\overline{\psi ^{h }}\frac{ih }{2}\partial _{xx}\psi ^{h
}dx=0.
\end{eqnarray*}%
Next, we shall show that $\left\Vert h \partial _{x}\psi ^{h }\right\Vert^2
_{L_{x}^{2}}$ remains bounded during the second splitting step \eqref{ts2}: Recall that
$\Upsilon _{d}^{h }\left( x,t\right) $ is in fact independent of $t
$, due to the fact that $\frac{d}{dt}\mu _{jk}^{h }=0
$ in this step. Since
\begin{equation*}
\partial _{t}{}_{x}\psi ^{h } =-\frac{i}{h }\partial
_{x}\Upsilon _{d}^{h }\left( x,t\right) \psi ^{h }-\frac{i}{h
}\Upsilon _{d}^{h }\left( x,t\right) \partial _{x}\psi ^{h },
\end{equation*}%
we find
\begin{eqnarray*}
&&\frac{d}{dt}\int h ^{2}\left\vert \partial _{x}\psi
^{h }\right\vert ^{2}dx=h ^{2}\int \partial _{tx}\overline{\psi
^{h }}\partial _{x}\psi ^{h }+\partial _{x}\overline{\psi ^{h
}}\partial _{tx}\psi ^{h }dx
\\
&=&h \int i\left( \partial _{x}\Upsilon _{d}^{h }\left( x,t\right)
\overline{\psi ^{h }}+\Upsilon _{d}^{h }\left( x,t\right) \partial
_{x}\overline{\psi ^{h }}\right) \partial _{x}\psi ^{h }-i\partial
_{x}\overline{\psi ^{h }}\Big( \partial _{x}\Upsilon _{d}^{h
}\left( x,t\right) \psi ^{h }+\Upsilon _{d}^{h }\left( x,t\right)
\partial _{x}\psi ^{h }\Big) dx
\\
&=&-2h \int \mathop{\rm Im}\left( \partial _{x}\Upsilon _{d}^{h
}\left( x,t\right) \overline{\psi ^{h }}\partial _{x}\psi ^{h}\right) dx
\leq 2\left\Vert \partial _{x}\Upsilon _{d}^{h }\left(
x,t\right) \right\Vert _{L_{x}^{\infty }}\left\Vert \psi ^{h
}\right\Vert _{L_{x}^{2}}\left\Vert h \partial _{x}\psi ^{h
}\right\Vert _{L_{x}^{2}},
\end{eqnarray*}%
where
\begin{equation*}
\left\Vert \partial _{x}\Upsilon _{d}^{h }\left( x,t\right) \right\Vert
_{L_{x}^{\infty }}=\left\Vert
\sum\limits_{j=0}^{J-1}\sum\limits_{k=0}^{K-1}\partial _{x}V\left(
x,y_{j}\right) \mu _{jk}^{h }\Delta y\Delta \eta \right\Vert
_{L_{x}^{\infty }}\leq \left\Vert \partial _{x}V\left( x,y\right)
\right\Vert _{L^{\infty }}.
\end{equation*}%
Since $\left\Vert \psi ^{h}\right\Vert _{L_{x}^{2}}=1$ is conserved by our scheme, we thus have
\begin{equation*}
\frac{d}{dt}\left\Vert h \partial _{x}\psi ^{h }\right\Vert
_{L_{x}^{2}}\leq C_{0},
\end{equation*}%
where $C_{0}=\left\Vert \partial _{x}V\left( x,y\right)
\right\Vert _{L^{\infty }}$ is some constant independent of $h$. Hence, in
the second splitting step \eqref{ts2} one has
\begin{equation*}
\left\Vert h \partial _{x}\psi ^{h ,n+1}\right\Vert
_{L_{x}^{2}}\leq \left\Vert h \partial _{x}\psi ^{h ,\ast
}\right\Vert _{L_{x}^{2}}+C_{0}\Delta t=\left\Vert h \partial _{x}\psi
^{h ,n}\right\Vert _{L_{x}^{2}}+C_{0}\Delta t,
\end{equation*}%
where we used the fact that $\left\Vert h \partial _{x}\psi
^{h }\right\Vert _{L_{x}^{2}}$ is conserved during \eqref{ts1}. In summary, this yields
\begin{equation*}
\left\Vert h \partial _{x}\psi ^{h ,n}\right\Vert _{L_{x}^{2}}\leq
\left\Vert h \partial _{x}\psi _{\rm in}^{h }\right\Vert
_{L_{x}^{2}}+C_{0}t_{n}\leq \left\Vert h \partial _{x}\psi
_{\rm in}^{h }\right\Vert _{L_{x}^{2}}+C_{0}T,
\end{equation*}%
where the right hand side is some constant independent of $h $ thanks
to the assumption on initial data \eqref{a2}. This shows that $%
\psi ^{h ,n}$ is $h-$oscillatory for any $n\in \N$, with $0\leq
t_{n}\leq T$ and the result follows from the arguments in the previous section.
{{
\begin{remark}
Note that all estimates above remain valid in the context of a Strang splitting scheme.
\end{remark}
}}


\section{Numerical examples}\label{sec:examples}

In this final section, we shall report on a few numerical examples, which illustrate the validity of our algorithm and meshing strategy.
To this end, we
choose an interaction potential of the form \[
V\left( x,y\right) =\frac{%
\left( x+y\right) ^{2}}{2},\] and solve the one-dimensional SLE system on the interval $x\in %
\left[ -\pi ,\pi \right] $ and $y,\eta \in \left[ -2\pi ,2\pi \right] $ with
periodic boundary conditions.

\begin{example}[$\Delta t$ independent of $h$]  \label{ex1n}
We choose initial conditions for the SLE system \eqref{SL-1} as follows:
\begin{equation*}
\psi _{\rm in}\left( x\right) =\exp\left({-25\left( x+0.2\right) ^{2}}\right) \exp\left({\frac{-i\ln
\left( 2\cosh \left( 5\left( x+0.2\right) \right) \right) }{5h }}\right),
\end{equation*}
and
\[
\ \mu
_{\rm in}\left( y,\eta \right) =\left\{
\begin{array}{cc}
C_{\rm N}\, \exp\left({-\frac{1}{1-y^{2}}}\right) \exp\left({-\frac{1}{1-\eta ^{2}}}\right), & \  \text{for $\left\vert y\right\vert <1$, $\left\vert \eta \right\vert <1$} \\
0, & \text{otherwise}.%
\end{array}
\right.
\]
Here, $C_{\rm N}>0$ is the normalization factor such that $\iint_{%
\mathbb{R}^{2}}\mu _{in} dy\, d\eta $=1.
{{Here we use the time-splitting method with spectral-upwind scheme (i.e., with an upwind scheme for the Liouville's equation).}} For $h =$ $\frac{1}{256},\frac{1}{1024},%
\frac{1}{4096}$, we fix the stopping time $T=0.5$ and choose $\Delta x=\frac{%
2\pi h }{16}$, $\Delta y=\Delta \eta =\frac{4\pi }{128}$.
For each choice of $%
h $, we shall solve the SLE system first with
$\Delta t\text{ independent of }h$ and, second, with $
\Delta t=o\left( h \right)$. To be more specific, we compare the two cases where $\Delta t=0.01$ and $\Delta t=%
\frac{h }{10}$. It can be observed from Figure \ref{ex1}, that the macroscopic position and current densities associated to the solution of Schr\"{o}dinger's equation agree well with each other.

\begin{figure}[tbph]
\textrm{{} }
\par
\textrm{\setcaptionwidth{4.6in} }
\par
\begin{center}
\textrm{%
\subfloat[Position Density]{\includegraphics[scale=0.37]
{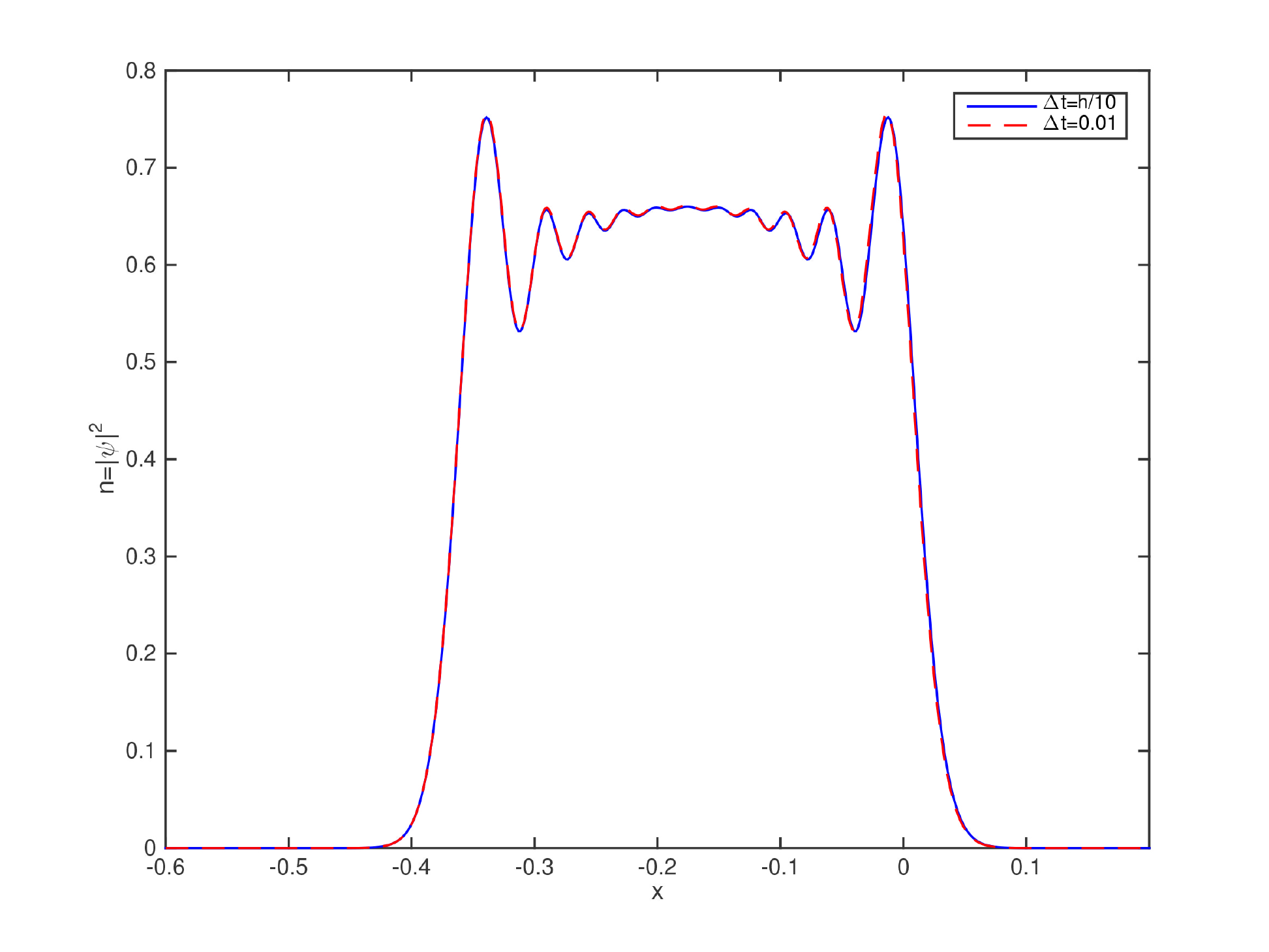}}
\subfloat[Current Density]{\includegraphics[scale=0.37]
{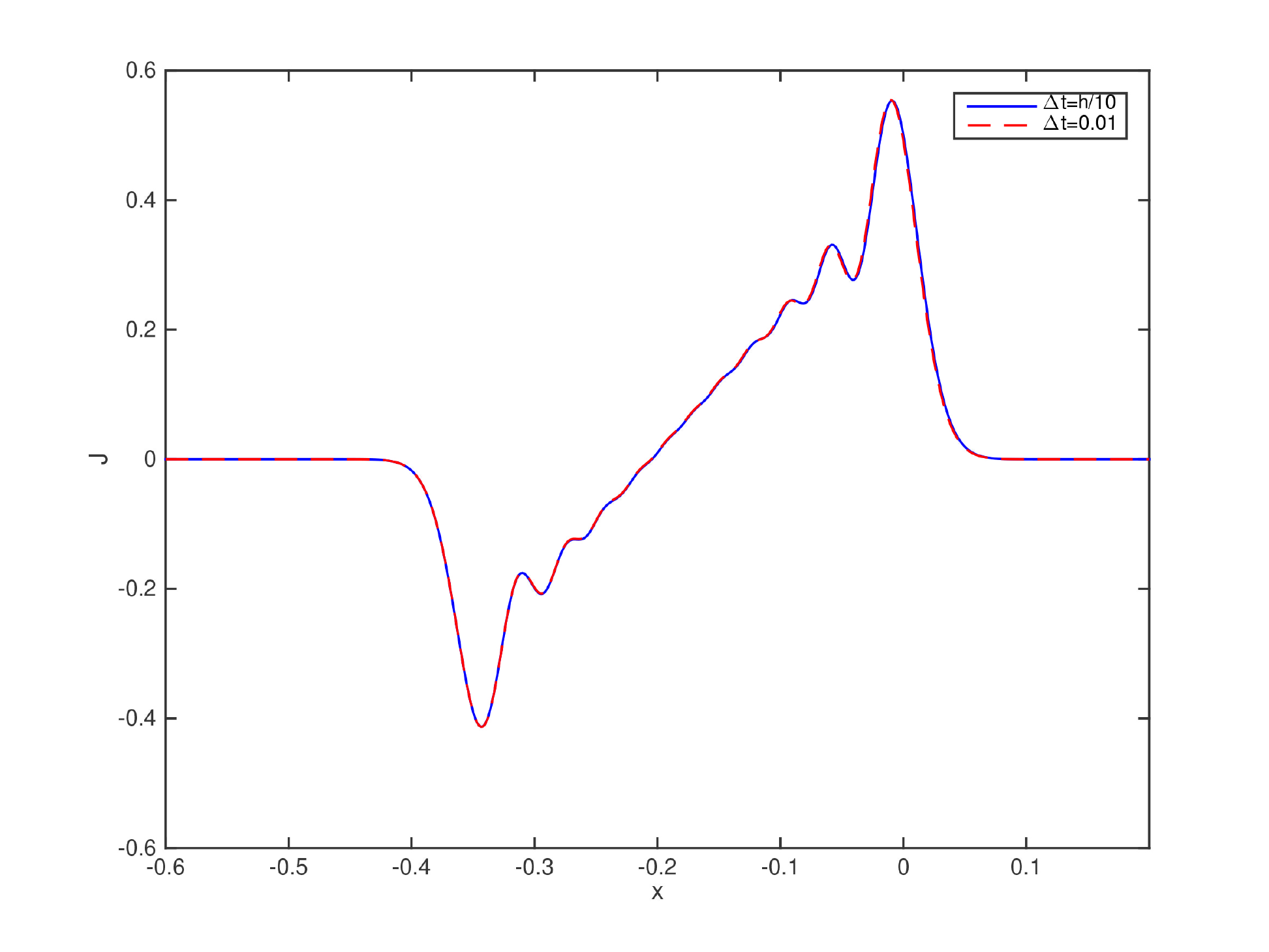}} 
}

\textrm{%
\subfloat[Position Density]{\includegraphics[scale=0.37]
{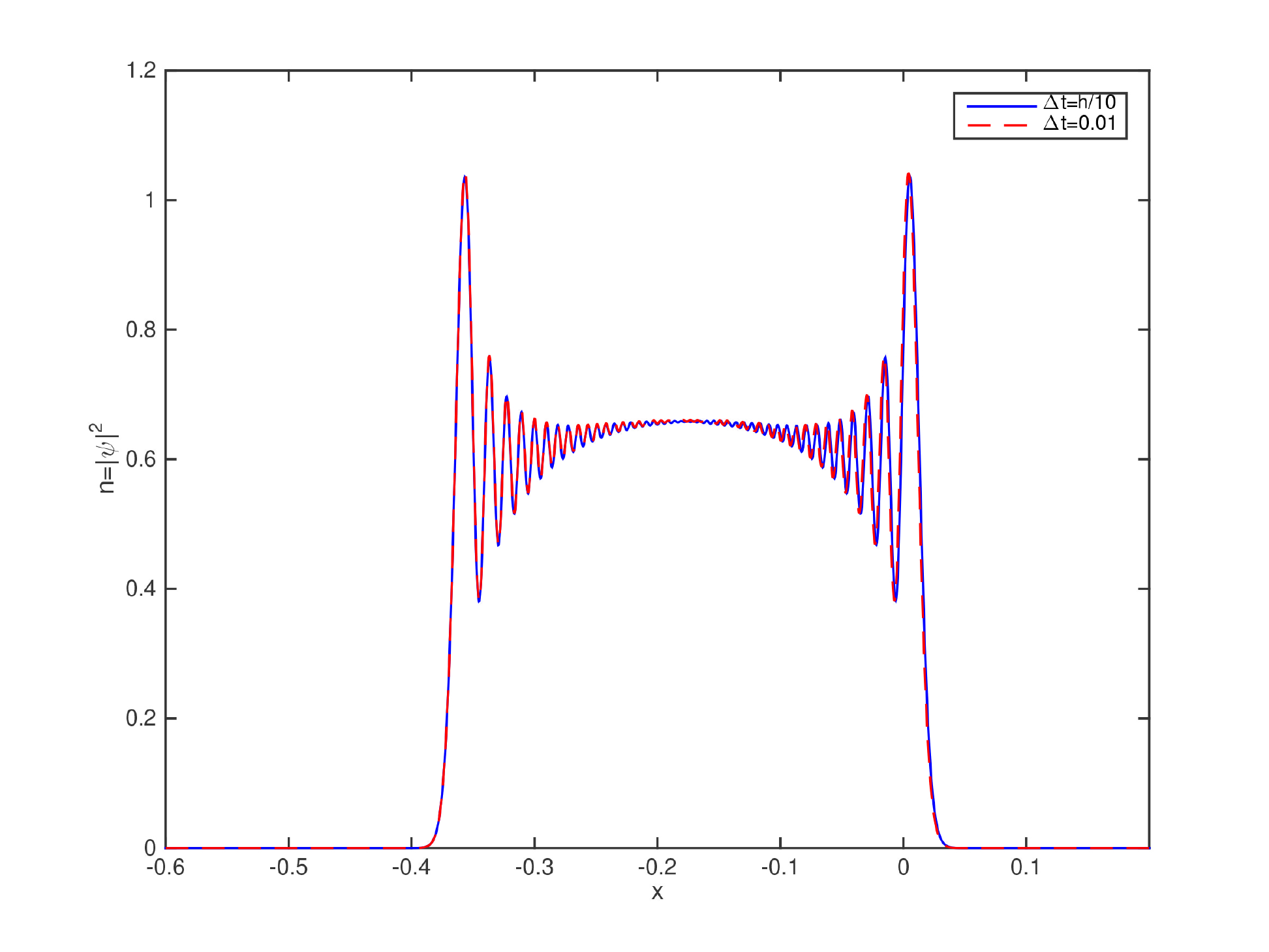}}
\subfloat[Current Density]{\includegraphics[scale=0.37]
{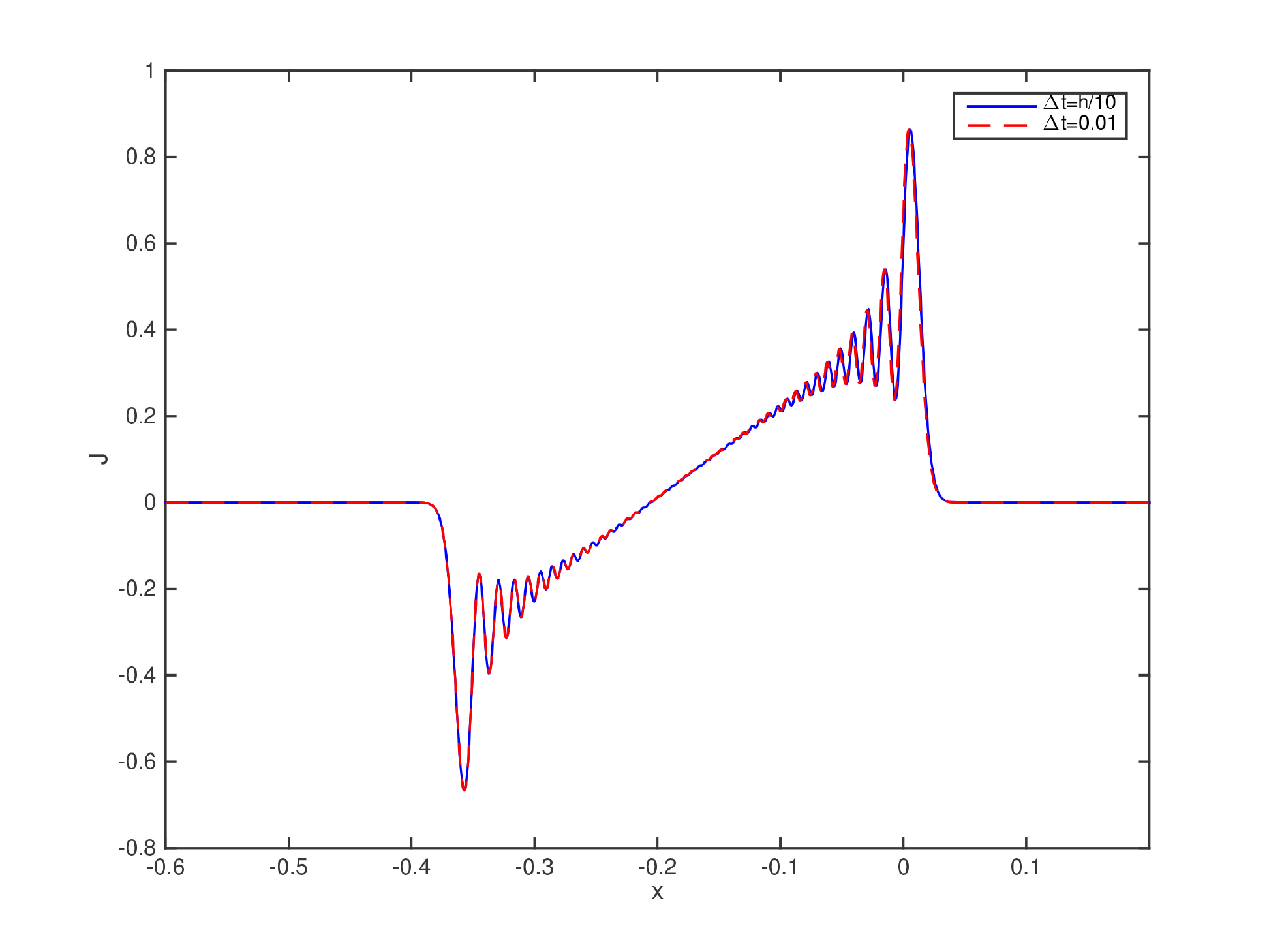}} }

\textrm{%
\subfloat[Position Density]{\includegraphics[scale=0.37]
{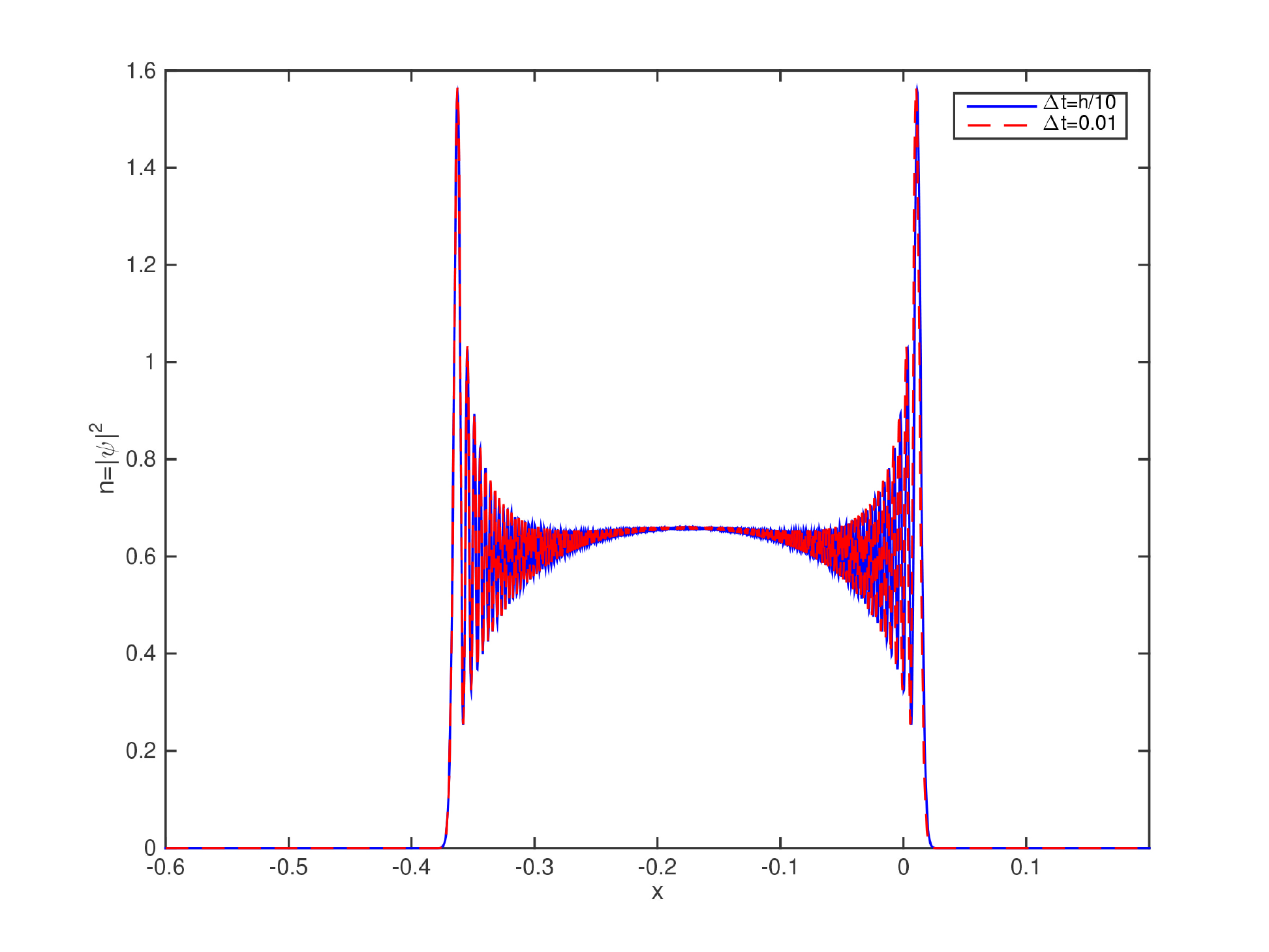}}
\subfloat[Current Density]{\includegraphics[scale=0.37]
{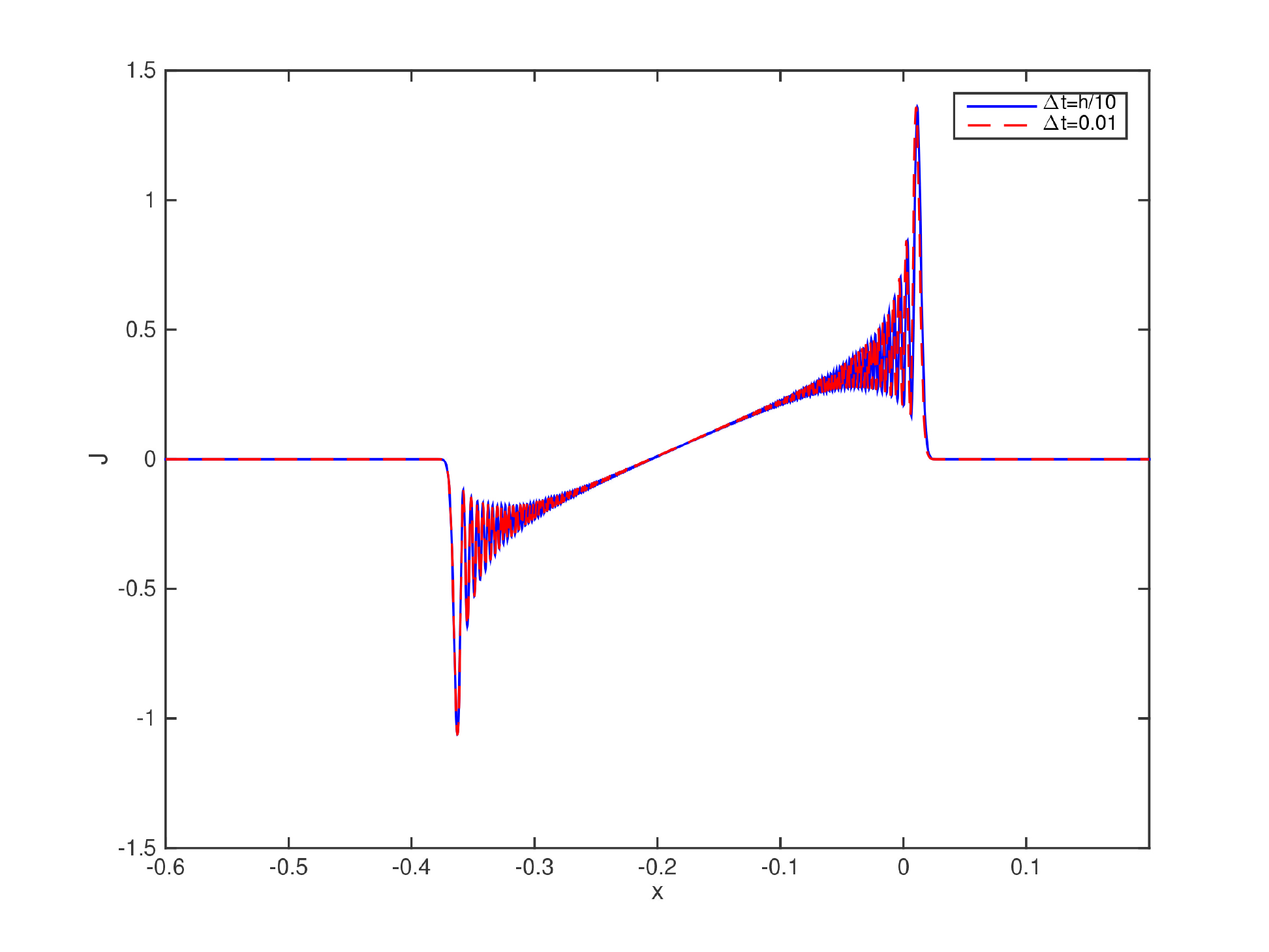}} }
\end{center}
\setcaptionwidth{4.6in}
\caption{Numerical solutions at $T=0.5$ in Example \ref{ex1n} computed by the
time-splitting method using different meshing strategies. First row: $h =\frac{1}{256}$; Second row: $h =\frac{1}{1024}$; Third
row: $h =\frac{1}{4096}$. }
\label{ex1}
\end{figure}

\textrm{In addition, we compare the numerical values of $\mu$ computed by $\Delta t=0.01$ and $\Delta t=\frac{h }{10}$ (denoted as $\mu_1$ and $\mu_2$, respectively). As shown in Table \ref{table1}, the error is insensitive in $h$, showing a uniform in $h$ convergence in $\Delta t$. }

\begin{table}[ht]
\centering
\setcaptionwidth{4.6in}

\begin{tabular}{|c|c|c|c|}
\hline
$h $ & $1/256$ & $1/1024$ & $1/4096$ \\ \hline
$\frac{\left\Vert \mu_1-\mu_2 \right\Vert_{\ell^2}}{\left\Vert \mu_2 \right\Vert_{\ell^2}} $ & 1.65e-03 & 1.69e-03 & 1.70e-03 \\ \hline
\end{tabular}%
\caption{The relative $\ell^2-$difference (defined as $\frac{\left\Vert \mu_1-\mu_2 \right\Vert_{\ell^2}}{\left\Vert \mu_2 \right\Vert_{\ell^2}} $) for various $h$.}
\label{table1}
\end{table}
\end{example}

\begin{example}[Numerical error as $h $ decreases] \label{ex2}
In this example, we choose the same initial data for $\mu_{\rm in}$ as before and
\begin{equation*}
\psi _{\rm in}\left( x\right) =\exp\left({-5\left( x+0.1\right) ^{2}}\right)\exp \left({\frac{i\sin x}{%
h }}\right).
\end{equation*}%
Now, we fix $\Delta t=$ $0.01$, a stopping time $T=0.4$, and $\Delta y=\Delta \eta =%
\frac{4\pi }{128}$. We choose $\Delta x=\frac{2\pi
h }{16}$, for $h =\frac{1}{64},\frac{1}{128},\frac{1}{256},%
\frac{1}{512},\frac{1}{1024},\frac{1}{2048}$, respectively. The reference solution is computed with $\Delta
t=\frac{h }{10}$. From the $\ell^2$-error plotted in Figure \ref{delta conv},
one can see that although the error in the wave function increases as $%
h $ decreases, the error for the position density $| \psi^h |
^{2}$ as well as for the macroscopic quantity $\mu $ does not change noticeably. This
shows that $h-$independent time steps can be taken to accurately obtain
physical observables, but not the wave function  $\psi^h$ itself.
\end{example}

\begin{figure}[tbph]
{}
\par
\setcaptionwidth{4.6in}
\par
\begin{center}
\includegraphics[scale=0.44]
{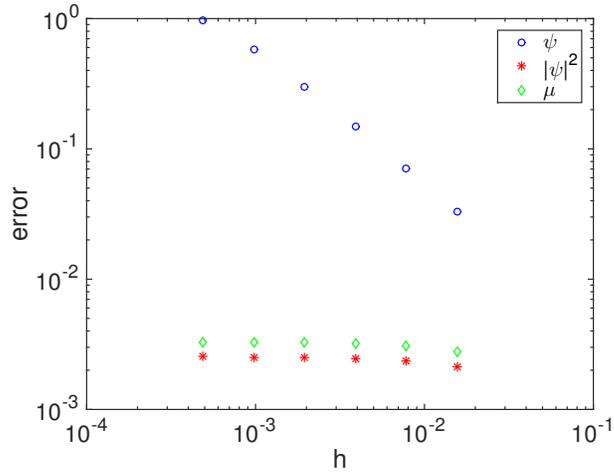}
\end{center}
\caption{Example \ref{ex2}: $\ell^2-$errors of the wave function $\psi^h$, position density $|\psi^h|^2$ and $\mu$ for various $h$. Fix $\Delta t=$ $0.01$. For $h =\frac{1}{64},\frac{1}{%
128},\frac{1}{256},\frac{1}{512},\frac{1}{1024},\frac{1}{2048}$, choose $%
\Delta x=\frac{2\pi h }{16}$ respectively. The
reference solution is computed with $\Delta t=\frac{h }{10}$. }
\label{delta conv}
\end{figure}

\begin{example}[Convergence in time]  \label{ex3}
Finally, to examine the convergence in time of our scheme, let the initial data be as in the example before.
Fix $h =$ $\frac{1}{8192}$, a stopping time $T=0.4$ and a spatial discretization with $\Delta x=\frac{2\pi
h }{16}$, $\Delta y=\Delta \eta =\frac{4\pi }{128}$. Choose $\Delta t=%
\frac{0.4}{32},\frac{0.4}{64},\frac{0.4}{128},\frac{0.4}{256},\frac{0.4}{512}%
,\frac{0.4}{1024}$. The reference solution is computed with $\Delta t=\frac{%
0.4}{81920}$. The $\ell^2$-error is plotted in Figure \ref{time conv}, which
shows first order accuracy in time of our scheme. Again, we see that the wave function exhibits errors
several orders of magnitude larger than the physical observable densities.
\end{example}

\begin{figure}[tbph]
{}
\par
\setcaptionwidth{4.6in}
\par
\begin{center}
\includegraphics[scale=0.44]
{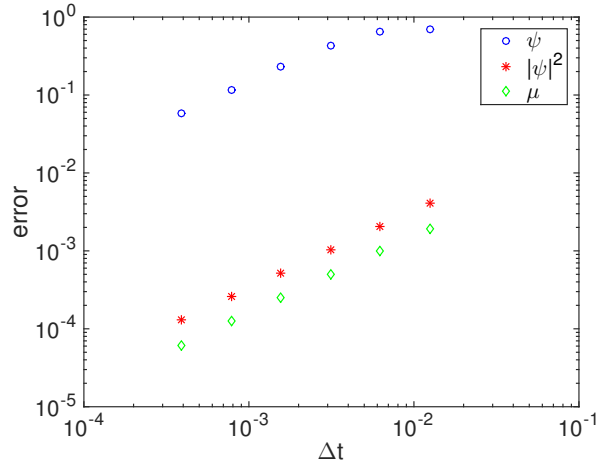}
\end{center}
\caption{Example \ref{ex3}: $\ell^2-$errors of the numerical solutions for various $\Delta t$ and fixed $h =\frac{1}{8192}$, $\Delta x=\frac{2%
\pi h }{16}$, $\Delta y=\Delta \eta =\frac{4\pi
}{128}$. It shows first order convergence of the scheme in time. The reference solution is computed with $\Delta t=\frac{0.4}{81920}$.}
\label{time conv}
\end{figure}

\bigskip

\bibliographystyle{abbrv}
\bibliography{ehrenfest}

\begin{thebibliography}{10}

\bibitem{Bao:2002fy}
W.~Bao, S.~Jin, and P.~A. Markowich.
\newblock {On Time-Splitting Spectral Approximations for the {S}chr{\"o}dinger
  Equation in the Semiclassical Regime}.
\newblock {\em J. Comput. Phys.}, 175(2):487--524, 2002.

\bibitem{bao2003numerical}
W.~Bao, S.~Jin, and P.~A. Markowich.
\newblock Numerical study of time-splitting spectral discretizations of
  nonlinear {S}chr{\"o}dinger equations in the semiclassical regimes.
\newblock {\em SIAM J. Sci. Comput.}, 25(1):27--64, 2003.

\bibitem{szepessy2015}
C.~Bayer, H.~Hoel, A.~Kadir, P.~Plech{\'a}{\v{c}}, M.~Sandberg, and
  A.~Szepessy.
\newblock {C}omputational error estimates for {B}orn--{O}ppenheimer molecular
  dynamics with nearly crossing potential surfaces.
\newblock {\em Appl. Math. Res. Express}, 2015(2):329--417, 2015.

\bibitem{billing2003quantum}
G.~D. Billing.
\newblock {\em {The Quantum Classical Theory}}.
\newblock Oxford University Press, 2003.

\bibitem{tdscf1}
R.~H. Bisseling, R.~Kosloff, R.~B. Gerber, M.~A. Ratner, L.~Gibson, and
  C.~Cerjan.
\newblock Exact time-dependent quantum mechanical dissociation dynamics of
  {I2He}: {C}omparison of exact time-dependent quantum calculation with the
  quantum time-dependent self-consistent field ({TDSCF}) approximation.
\newblock {\em J. Chem. Phys.}, 87(5):2760--2765, 1987.

\bibitem{bornemann1996quantum}
F.~A. Bornemann, P.~Nettesheim, and C.~Sch{\"u}tte.
\newblock Quantum-classical molecular dynamics as an approximation to full
  quantum dynamics.
\newblock {\em J. Chem. Phys.}, 105(3):1074--1083, 1996.

\bibitem{Carles1}
R.~Carles.
\newblock On {F}ourier time-splitting methods for nonlinear {S}chr\"{o}dinger
  equations in the semiclassical limit.
\newblock {\em SIAM J. Numerical Anal.}, 51(6):3232--3258, 2013.

\bibitem{Carles2}
R.~Carles and C.~Gallo.
\newblock On {F}ourier time-splitting methods for nonlinear {S}chr\"{o}dinger
  equations in the semi-classical limit ii. {A}nalytic regularity.
\newblock {\em Numer. Math.}, 136(1):315--342, 2017.

\bibitem{drukker1999basics}
K.~Drukker.
\newblock Basics of surface hopping in mixed quantum/classical simulations.
\newblock {\em J. Comput. Phys.}, 153(2):225--272, 1999.

\bibitem{ehrenfest1927bemerkung}
P.~Ehrenfest.
\newblock Bemerkung {\"u}ber die angen{\"a}herte {G}{\"u}ltigkeit der
  klassischen {M}echanik innerhalb der {Q}uantenmechanik.
\newblock {\em Z. Phys. A, Hadrons Nucl.}, 45(7):455--457, 1927.

\bibitem{wigner4}
P.~Gerard, P.~A. Markowich, N.~J. Mauser, and F.~Poupaud.
\newblock Homogenization limits and {W}igner transforms.
\newblock {\em Comm. Pure Appl. Math.}, 50(4):323--379, 1997.

\bibitem{Jin99}
S.~Jin.
\newblock Efficient asymptotic-preserving ({AP}) schemes for some multiscale
  kinetic equations.
\newblock {\em SIAM J. Sci. Comput.}, 21(2):441--454 (electronic), 1999.

\bibitem{Jin-review}
S.~Jin.
\newblock Asymptotic preserving ({AP}) schemes for multiscale kinetic and
  hyperbolic equations: a review.
\newblock {\em Riv. Math. Univ. Parma (N.S.)}, 3(2):177--216, 2012.

\bibitem{JMS}
S.~Jin, P.~Markowich, and C.~Sparber.
\newblock Mathematical and computational methods for semiclassical
  {S}chr\"odinger equations.
\newblock {\em Acta Numer.}, 20:121--209, 2011.

\bibitem{Jin:2017bh}
S.~Jin, C.~Sparber, and Z.~Zhou.
\newblock {On the classical limit of a time-dependent self-consistent field
  system: {A}nalysis and computation}.
\newblock {\em Kinet. Relat. Models}, 10(1):263--298, Mar. 2017.

\bibitem{tdscf3}
Z.~Kotler, E.~Neria, and A.~Nitzan.
\newblock Multiconfiguration time-dependent self-consistent field
  approximations in the numerical solution of quantum dynamical problems.
\newblock {\em Comput.Phys. Comm.}, 63(1):243--258, 1991.

\bibitem{tdscf2}
Z.~Kotler, A.~Nitzan, and R.~Kosloff.
\newblock Multiconfiguration time-dependent self-consistent field approximation
  for curve crossing in presence of a bath. {A} fast {F}ourier transform study.
\newblock {\em Chem. Phys. Lett.}, 153(6):483--489, 1988.

\bibitem{leveque1992numerical}
R.~J. LeVeque.
\newblock {\em Numerical Methods for Conservation Laws}.
\newblock Birkh{\"a}user Basel, 1992.

\bibitem{lions1993mesures}
P.-L. Lions and T.~Paul.
\newblock Sur les mesures de {W}igner.
\newblock {\em Rev. Mat. Iberoamericana}, 9(3):553--618, 1993.

\bibitem{tdscf4}
N.~Makri and W.~H. Miller.
\newblock Time-dependent self-consistent field ({TDSCF}) approximation for a
  reaction coordinate coupled to a harmonic bath: Single and multiple
  configuration treatments.
\newblock {\em J. Chem. Phys.}, 87(10):5781--5787, 1987.

\bibitem{markowich1993classical}
P.~A. Markowich and N.~J. Mauser.
\newblock The classical limit of a self-consistent quantum-{V}lasov equation in
  3d.
\newblock {\em Math. Models Methods Appl. Sci.}, 3(01):109--124, 1993.

\bibitem{marx2000ab}
D.~Marx and J.~Hutter.
\newblock Ab {I}nitio molecular dynamics: {T}heory and implementation.
\newblock {\em Modern Meth. Algorithm. Quant. Chem.}, 1(141):301--449, 2000.

\bibitem{marx2009ab}
D.~Marx and J.~Hutter.
\newblock {\em {Ab Initio Molecular Dynamics: {B}asic Theory and Advanced
  Methods}}.
\newblock Cambridge University Press, 2009.

\bibitem{schutte1999singular}
C.~Sch{\"u}tte and F.~A. Bornemann.
\newblock On the singular limit of the quantum-classical molecular dynamics
  model.
\newblock {\em SIAM J. Appl. Math.}, 59(4):1208--1224, 1999.

\bibitem{SMM2003}
C.~Sparber, P.~Markowich, and N.~Mauser.
\newblock {Wigner functions versus WKB methods in multivalued geometrical
  optics}.
\newblock {\em Asymptot. Anal.}, 33(2):152--187, 2003.

\bibitem{szepessy2011}
A.~Szepessy.
\newblock Langevin molecular dynamics derived from {E}hrenfest dynamics.
\newblock {\em Math. Models Methods Appl. Sci.}, 21(5):2289--2334, 2011.

\bibitem{taylor1996partial}
M.~E. Taylor.
\newblock {\em {Partial Differential Equations: Basic Theory}}.
\newblock Applied mathematical sciences. Springer, 1996.

\bibitem{tully1998mixed}
J.~C. Tully.
\newblock Mixed quantum--classical dynamics.
\newblock {\em Faraday Discussions}, 110:407--419, 1998.

\bibitem{wigner1932quantum}
E.~Wigner.
\newblock On the quantum correction for thermodynamic equilibrium.
\newblock {\em Phys. Rev.}, 40(5):749, 1932.

\end{thebibliography}

\end{document}